\newif\ifmaster
\masterfalse 
\documentclass[12pt]{amsart}

\title
[Monodromy of a system of rank $p^m$]
{The monodromy representation of 
a hypergeometric system in $m$ variables of rank $p^m$}

\author[Kaneko J.]{Kaneko Jyoichi}
\address[Kaneko]{
Department of Mathematical Sciences,
University of the Ryukyus,
Nishihara, Okinawa, 903-0213, Japan
}
\email{kaneko@math.u-ryukyu.ac.jp}

\author[Matsumoto K.]{Matsumoto Keiji}
\address[Matsumoto]{
Department of Mathematics,\\
Faculty of Science,\\
Hokkaido University,\\
Sapporo 060-0810, Japan
}
\email{matsu@math.sci.hokudai.ac.jp}

\author[Ohara K.]{Ohara Katsuyoshi}
\address[Ohara]{
Faculty of Mathematics and Physics\\
Kanazawa University,\\
Kanazawa 920-1192, Japan\\
}
\email{ohara@se.kanazawa-u.ac.jp
}

\author[Terasoma T.]{Terasoma Tomohide}
\address[Terasoma]{
Faculty of Science and Engineering,
Hosei University,
Koganei, Tokyo 184-8584, Japan
}
\email{terasoma@hosei.ac.jp}

\keywords{Hypergeometric functions, Hypergeometric system, 
Monodromy representation}
\subjclass[2020]{33C70, 32S40}

\date{\today}

\usepackage{amsmath}	

\usepackage{amssymb}
\usepackage{amsthm}
\usepackage{graphicx}
\usepackage{color}
\usepackage[all]{xy}
\usepackage{mathtools}
\newcommand{\C}{\mathbb C}
\newcommand{\D}{\mathbb D}

\newcommand{\N}{\mathbb N}

\newcommand{\Q}{\mathbb Q}
\newcommand{\Z}{\mathbb Z}

\newcommand{\cF}{\mathcal F}

\newcommand{\cH}{\mathcal H}

\newcommand{\pa}{\partial}
\newcommand{\ex}{\mathbf{e}}

\renewcommand{\a}{\alpha} 
\renewcommand{\b}{\beta} 
\newcommand{\g}{\gamma}
\newcommand{\G}{\varGamma}

\newcommand{\e}{\varepsilon}
\newcommand{\f}{\varphi}
\renewcommand{\l}{\lambda}

\newcommand{\z}{\zeta}

\newcommand{\tr}{\;^t}
\newcommand{\Tr}{\mathrm{tr}}

\newcommand{\diag}{\mathrm{diag}}
\newcommand{\one}{\mathbf{1}}

\newcommand{\comment}[1]{{}}
\newcommand{\ds}[1]{\displaystyle{#1}}

\newcommand{\loc}{\operatorname{loc}}
\newcommand{\Ad}{\operatorname{Ad}}

\renewcommand{\Im}{\operatorname{Im}}
\newcommand{\Ker}{\operatorname{Ker}}
\newcommand{\vcc}{{\bf f}_J}
\newcommand{\vccp}{{\bf f}_{J'}}
\newcommand{\vccpp}{{\bf f}_{J''(p)}}
\newcommand{\vccpj}{{\bf f}_{J'(j)}}
\allowdisplaybreaks

\theoremstyle{plain}
\newtheorem{theorem}{Theorem}[section]
\newtheorem{proposition}[theorem]{Proposition}
\newtheorem{lemma}[theorem]{Lemma}
\newtheorem{cor}[theorem]{Corollary}
\newtheorem{fact}[theorem]{Fact}

\theoremstyle{definition}

\newtheorem{remark}[theorem]{Remark}
\newtheorem{definition}[theorem]{Definition}

\makeatletter
\@addtoreset{equation}{section}
\@addtoreset{figure}{section}
\@addtoreset{table}{section}
\makeatother

\mathtoolsset{showonlyrefs=true}

\begin{document}

\maketitle
\begin{abstract}
We study the monodromy representation of the hypergeometric
system $\cF_{C}^{p,m}(a,B)$ in $m$ variables of rank $p^m$ with parameters 
$a$ and $B$. This system can be regarded as a multi-variable model of 
the generalized hypergeometric equation of rank $p$.  
We construct $m+1$ loops which generate 
the fundamental group of the complement of the singular locus  
of $\cF_{C}^{p,m}(a,B)$, 
and we show that they satisfy certain relations 
as elements of the fundamental group. 
We produce circuit matrices along these loops with respect to 
a fundamental system of solutions to $\cF_C^{p,m}(a,B)$ under certain 
non-integrality conditions on parameters $a$ and $B$.
\end{abstract}
\tableofcontents
\section{Introduction} 
\subsection{Introduction and results of \cite{KMOT}}
Let $p\geq 2, m \geq 1$ be integers, $a=\tr(a_1,\dots,a_p)$ be an element of $\C^p$ and
$$
B=(b_{i,j})_{\substack{{1\le i\le p}\\ {1 \le j\le m}}}
=\begin{pmatrix}
b_{1,1} &  & b_{1,m}\\
\vdots & \cdots & \vdots\\
b_{p-1,1} &  & b_{p-1,m}\\
1 &  & 1
\end{pmatrix}
$$
be a $p\times m$ matrix satisfying the condition
$$
b_{j,k}\notin \Z\ 
(1\le j\le p-1, 1\le k\le m),\quad
b_{p,k}=1\  (1\le k\le m).
$$
A parameter $B$ is said to be standard, if it satisfies the above 
condition. In \cite{KMOT}, we define a hypergeometric series 
$F_C^{p,m}(a,B;x)$ in $m$ variables
 $x_1$, $\dots$, $x_m$ with parameters $a$ and $B$ by

\begin{align}
\label{eq:HGS ser}
 &\quad F_C^{p,m}(a,B;x)
\\ \nonumber
=&
\sum_{(n_1,\dots,n_m)\in \N^m}
\frac{(a_1,n_1+\cdots +n_m)\cdots (a_p,n_1+\cdots +n_m)}
{\prod_{k=1}^{m}\{(b_{1,k},n_k)\cdots (b_{p-1,k},n_k)(b_{p,k},n_k) \}}
x_1^{n_1}\cdots x_m^{n_m}.
\end{align}
There we have shown that the 
hypergeometric series $F_C^{p,m}(a,B;x)$
satisfies a system of differential equations.
\begin{proposition}[System of differential equations, {\cite[Lemma 4.1]{KMOT}}]
\label{lem:HGDE}
For $k=1,\dots,m$ and parameters $a$ and $B$, 
we define differential operators $\ell_k=\ell(a,B)$
by 
\begin{align*}
\ell_k(a,B)= &(b_{1,k}-1+\theta_k)\cdots (b_{p-1,k}-1+\theta_k)
\theta_k
\\
\nonumber
&-x_k(a_1+\theta_1+\cdots+\theta_m)\cdots
(a_p+\theta_1+\cdots+\theta_m),
\end{align*}
where $\theta_k$ denotes the Euler operator $x_k\frac{\pa}{\pa x_k}$.  
Then the series $F_C^{p,m}(a,B;x)$ satisfies the following 
differential equations:
\begin{align*}
\label{differential equations first time}
\ell_k(a,B)
F_C^{p,m}(a,B;x)=0
\end{align*}
for $k=1,\dots,m$.
\end{proposition}
The system of these differential equations 
is denoted by $\cF_C^{p,m}(a,B)$.
Its rank and singular locus are studied in \cite{MO} and \cite{KMOT}; 
they are given as follows. 

\begin{proposition}[{\cite[Proposition 4.4, Theorems 5.1 and 5.2]{KMOT}}] 
\label{result of part I}
We suppose some non-integrality conditions on parameters $a,B$.  
\begin{enumerate}
 \item 
The rank of the system $\cF_C^{p,m}(a,B)$ is $p^m$.
For $J=(j_1,\dots,j_m )\in \{1,\dots, p\}^m$,
we have linearly independent $p^m$ solutions $\Phi_J(a,B;x)$ 
to $\cF_C^{p,m}(a,B)$ around a point near to the origin  
(see Proposition \ref{fact:fund-solutions} in this paper for their 
explicit forms).
\item
The singular locus of $\cF_C^{p,m}(a,B)$ is 
$$S(x)=\{x\in \C^m\mid x_1\cdots x_m\cdot R(x)=0\},$$
where 
$R(x)=R(x_1,\dots,x_m)$ is an irreducible polynomial in $x_1,\dots,x_m$ 
of degree $p^{m-1}$ given by 
\begin{equation}
\label{eq:R(x)}
R(x_1,\dots,x_m)=\prod_{(i_1,\dots,i_m)\in (\Z/p\Z )^m}
(1-\z_p^{i_1}\sqrt[p]{x_1}-\cdots -\z_p^{i_m}\sqrt[p]{x_m}),
\end{equation}
where $\z_p=\exp(2\pi\sqrt{-1}/p)$ is a primitive $p$-th root  of unity.
\end{enumerate}
\end{proposition}

This paper is a continuation of the paper \cite{KMOT}.
We study the monodromy representation of the system
$\cF_C^{p,m}(a,B)$ using several properties of 
the fundamental group of the complement of the singular  locus $X=\C^m-S(x)$ of
this system.
Let $\e$ be a sufficiently small positive number
and set $\dot x=(\e^p,\dots,\e^p)\in X$.
We give a system of generators 
$\{\rho_0, \rho_1, \dots, \rho_m\}$
of the fundamental group
$\pi_1(X,\dot x)$
and show that they satisfy relations 
\eqref{eq:pi1-relations}.
Moreover, we construct a suitable basis $\{{\vcc}\}_J$,
called vcc (vanishing cycle component) basis, under 
some assumptions on vanishing cycles.
These assumptions will be proved in the forthcoming paper \cite{KMOT3}.
Using this basis, we explicitly give circuit matrices $M_0,M_1,\dots, M_m$ 
along  $\rho_0,\rho_1,\dots,\rho_m$ 
in Theorem \ref{theorem:main theorem}, which is our main theorem. 

Here we recall previous works 
 on monodromy representations of systems of hypergeometric differential 
equations related to this paper. 
As is  mentioned in \cite{KMOT}, the series $F_C^{p,m}(a,B;x)$ 
can be regarded as a multi-variable model of 
the generalized hypergeometric series $_{p}F_{p-1}$ 
just like the Lauricella hypergeometric series of type $C$
is one of multi-variable models of the Gauss hypergeometric series 
$_{2}F_{1}$.  
For the generalized hypergeometric 
differential equation satisfied by $_pF_{p-1}$,
its monodromy representation 
is studied 
in \cite{BH} and \cite{Ma2} (see also their references). 
Our main theorem is shown by induction on the number $m$ of variables, 
and \cite[Theorem 3.7]{Ma2} is used as its proof of the case $m=1$. 
For the Appell hypergeometric system satisfied by 
$F_4(a_1,a_2,b_1,b_2;x_1,x_2)$ 
( $=F_C^{2,2}(a,\begin{pmatrix}b_1 & b_2 \\ 1 & 1\end{pmatrix}
;x_1,x_2)$ in our notation),
its monodromy representation 
is studied in \cite{Kan} and \cite{T}.
In this case, the fundamental group of the complement of its singular locus 
is generated by three loops $\rho_0,\rho_1,\rho_2$ with relations 
$$\rho_1\cdot \rho_2=\rho_2\cdot \rho_1, \quad 
(\rho_0\cdot \rho_k)^2=(\rho_k\cdot \rho_0)^2\ (k=1,2).$$ 
They give explicit forms of the circuit matrices along these loops.
These results are generalized to those for the Lauricella 
hypergeometric  system of type $C$ 
satisfied by the series
$$F_C(a_1,a_2,b_1,\dots,b_m;x_1,\dots,x_m),$$
which is a special case $p=2$ for ours.
The structure of the fundamental group of the complement of its singular locus is studied 
in  \cite{GK}. Its monodromy representation is studied in 
\cite{Ma2} by induction on $m$, and in \cite{G2} by the intersection form 
on the twisted homology groups arising from an Euler type integral
expression.
For a hypergeometric system of rank $9$ in two variables satisfied by
$F\big(\begin{matrix} \tr a \\ \tr B \end{matrix};x_1,x_2\big)
=F_C^{3,2}(a,B;x_1,x_2)$, which is a two variable model of 
the generalized hypergeometric series $_3F_2$, 
the structure of the fundamental group of the complement of its singular locus is 
studied in \cite{KMO2}, and its monodromy representation is studied 
in \cite{KMO1}. 

We can regard these results for circuit matrices of several 
hypergeometric systems as special cases of 
our main theorem. 
In other words, we succeed in uniting the monodromy representations 
of several hypergeometric systems to general forms in our main theorem.

\subsection{Outline of this paper} 
\subsubsection{Generators and some relations of the fundamental group}

We study several relations between elements
in fundamental groups. 
Under the normal covering 
$$
\f:\C^m\ni z=(z_1,\dots,z_m)\mapsto (z_1^p,\dots,z_m^p)=(x_1,\dots,x_m)=x\in 
\C^m,
$$
with the Galois group $(\Z/p\Z )^m$,
the pull back $S(z)=\f^{-1}(S(x))$ of $S(x)$
is equal to
$$
S(z)=\{z_1\dots z_m=0\}\cup \bigcup_{(i_1,\dots,i_m)\in (\Z/p\Z)^m}
\{1-\z_p^{i_1}z_1-\cdots -\z_p^{i_p}z_1=0\}. 
$$
To study the fundamental groups of the spaces 
$X=\C^{m}-S(x)$ and $Z=\C^{m}-S(z)$ inductively on $m$,
we consider the singular loci $S'(x)$ and $S'(z)$ defined as the case of $m-1$.
We set $X'=\C^{m-1}-S'(x)$ and $Z'=\C^{m-1}-S'(z)$.
By considering the hyperplanes 
$\{z\in \C^m\mid z_m=\epsilon\}=\C^{m-1}\subset \C^m$
and $\{x\in \C^m\mid x_m=\epsilon^p\}=\C^{m-1}\subset \C^m$,
we introduce infinitesimal inclusions
\begin{equation*}
\iota_{\infty}:\pi_1(Z',\dot z')\to \pi_1(Z,\dot z),\quad
\iota_{\infty}:\pi_1(X',\dot x')\to \pi_1(X,\dot x).
\end{equation*}
of fundamental groups, where 
$\dot z'=(\e,\dots, \e)\in \C^{m-1}$,
$\dot x'=(\e^p,\dots, \e^p)\in \C^{m-1}$,
$\dot z=(\e,\dots, \e)\in \C^{m}$,
$\dot x=(\e^p,\dots, \e^p)\in \C^{m}$.
Let $\rho_0$ be a loop with base point $\dot x$ in $X$
 turning once positively around the divisor $R(x)=0$, and  
$\rho_k$ $(k=1,\dots,m)$ be that around the hyperplane $x_k=0$ 
near the origin, refer to 
\eqref{eq:loop0} and \eqref{eq:loops} 
for their definitions.  
Let $\rho_0'$ be a loop with base point $\dot x'$ in 
the $(m-1)$-dimensional space $X'$ defined as  the case of $m-1$.
Using the infinitesimal inclusions,
we show the equalities
$$
\iota_{\infty}(\rho_0')=\rho_0\cdot (\rho_m\rho_0\rho_m^{-1})\cdot (\rho_m^2\rho_0\rho_m^{-2})
\cdots(\rho_m^{p-1}\rho_0\rho_m^{1-p}),
$$
and
\begin{equation*}
\iota_{\infty}(\rho_0')\rho_m=\rho_m\iota_{\infty}(\rho_0'),\quad
(\rho_0\rho_k)^p=(\rho_k\rho_0)^p\ (1\le k\le m),
\end{equation*}
as elements in the fundamental group $\pi_1(X,\dot x)$
(Theorem \ref{th:reduction} (1),(2), Corollary 
\ref{circular relation}). 
Here we use the same notations of a loop and its homotopy class.  
We also show that the loops $\rho_0,\rho_1,\dots,\rho_m$ 
generate the fundamental group $\pi_1(X,\dot x)$ 
(Theorem \ref{generator of fund of X}).

\subsubsection{Representation spaces}
We consider an integral expression of the series $F_C^{p,m}(a,B;x)$. 
\begin{proposition}[Integral expression, {\cite[Theorem 3.1]{KMOT}}]
\label{fact:int-rep}
For a parameter $B$ satisfying the condition 
$$
a_p\notin \Z,\quad
a_j,\  
 b_{j,1},\dots,b_{j,m},\ a_j- b_{j,1}-\cdots-b_{j,m}\notin \Z\quad 
(1\le j\le p-1),
$$
we have the following integral expression of
$F_C^{p,m}(a,B;x)$:
\begin{align*}
F_C^{p,m}(a,B;x)
=
\prod_{j=1}^{p-1}
\frac{\G(1-a_j)}
{\G(1+e_j)\prod_{\substack{1\leq k\leq m}} \G(1-b_{j,k})
}
\int_{D}\!\!
\xi
\cdot \dfrac{dt}{t},
\end{align*}
where $e_j=b_{j,1}+\cdots+b_{j,m}-a_j-m, e_p=-a_p$, 
\begin{align*}
&\xi=\Big(\prod_{\substack{1\le j\le p\\1\le k\le m}} 
t_{j,k}^{1-b_{j,k}}\Big)\cdot  
\prod_{j=1}^{p} L_j(t)^{e_j},\quad
\dfrac{dt}{t}=\bigwedge_{j=1}^{p-1} 
\Big(\dfrac{dt_{j,1}}{t_{j,1}}\wedge \cdots \wedge 
\dfrac{dt_{j,m}}{t_{j,m}}\Big),
\\
&L_j(t,x)=1-t_{j,1}-\cdots-t_{j,m},\quad
\\
&L_p(t,x)=
1-\frac{x_1}{t_{1,1}\cdots t_{p-1,1}}-\cdots 
-\frac{x_m}{t_{1,m}\cdots t_{p-1,m}},
\\
\end{align*}
and
$D$ is the product of 
regularized twisted cycles 
defined in \cite[\S4]{G1}.
\end{proposition}
For an element $x$ in a small neighborhood of $\dot x$,
we can define twisted cycles $D_J$ indexed by $J\in \{1,\dots,p\}^m$
in the forthcoming paper \cite{KMOT3}.
The integral of $\xi\cdot dt/t$ over $D_J$ is denoted
by $\Phi^{\Gamma}_J(a,B;x)$.
The twisted homology 
group is identified with the $K$-linear span $\mathcal V$ of
$\{\Phi_J^\G(a,B;x)\}_J$,
where $K$ is the monodromy field defined in \eqref{def of monodromy field}. 
By the construction of twisted cycles, the space $\mathcal V$
is stable under circuit transformations.
In this sense, $K$ is the smallest field so that circuit transformations 
are well defined.


Let $M_0, M_1,\dots, M_m \in GL(K,p^m)$ be
the circuit matrices of the loops $\rho_0, \rho_1, \dots, \rho_m$ with 
respect to the basis $\{\Phi_J^\G(a,B;x)\}_J$.
Then we have
the simultaneous eigen decomposition 
\begin{equation}
\label{simaltaneous decomp int} 
\mathcal V=\bigoplus_J\mathcal V_J,\quad \mathcal V_J=K\Phi_J^\G(a,B;x),
\end{equation}
for the $m$ circuit transformations 
along $\rho_1, \dots, \rho_m$,
and the eigen decomposition $\mathcal V=\bigoplus_j\mathcal V_j$
with respect to the circuit transformation along $\rho_m$.

Let $\iota_{\infty}:\pi_1(X')\to \pi_1(X)$ be the 
infinitesimal inclusion.
Let $N_0$ be the circuit matrix of $\iota_{\infty}(\rho_0')$
on the space $\mathcal V$.
The space $\mathcal V_j$ is stable under the action of $N_0$.
The restriction of the action $N_0$ to $\mathcal V_j$
is denoted by $N_{0,j}$. 
The twisted homology group corresponding to
the system 
$\cF_C^{p,m-1}(a-(b_{j,m}-1)\one_p,B')$
is denoted by $\mathcal V'_j$, and its monodromy field is denoted by
$K'$.
The circuit matrix of $\rho_0'$
on the space $\mathcal V'_j$
is denoted by $M'_{0,j}$.
Then we have an
isomorphism 
\begin{equation}
\label{twisted hom and infinitesimal inclusion}
\mathcal V'_j\otimes_{K'} K \simeq \mathcal V_j,
\end{equation}
and via this isomorphism, the circuit matrix $N_{0,j}$
on $\mathcal V_j$ is equal to the circuit matrix $M'_{0,j}$ on
$\mathcal V'_j\otimes_{K'} K$.

\subsubsection{The circuit matrix $M_0$ as  an $H$-reflection}
The field $K$ is equipped with an involution $k\mapsto k^{\vee}$ 
defined by the sign change of the parameters $a,B$.
Similarly to $\mathcal V$, 
the $K$-linear span of
$\{\Phi_J^\G(-a,-B;x)\}_J$
is identified with 
the twisted homology 
group associated with $F_C^{p,m}(-a,-B;x)$, which is
denoted by $\mathcal V^\vee$.
We give a $K$-conjugate isomorphism 
$$
\mathcal V \ni \Phi_J^\G(a,B;x) \mapsto  \Phi_J^\G(-a,-B;x)\in \mathcal V^{\vee},
$$
and introduce an intersection form $\langle u,v\rangle$
for $u,v\in \mathcal V\simeq K^{p^m}$ under this isomorphism. 
It is $K$-skew symmetric, i.e.
$$\langle ku,v\rangle=k\langle u,v\rangle=\langle u,k^{\vee}v\rangle$$
for any $k\in K$, and  
satisfies the invariant property under the action of 
the circuit matrix  $M_{\rho}$:
$$
\langle uM_{\rho},vM_{\rho}\rangle=
\langle u,v\rangle.
$$
We define an $H$-reflection with respect a non-degenerate
$K$-skew symmetric matrix $H$ 
in \S \ref{subsec:H reflection}.
We easily show that an $H$-reflection $M$ is determined by
the $K$-skew symmetric form 
$\langle u,v\rangle=u H\tr v^\vee$, an element ${\bf v}$ of 
the one-dimensional space $\Im(M-I_n)$, and an inverse image ${\bf v}^*$ of ${\bf v}$ under $M-I_n$.
The vector ${\bf v}$ is called a reflection vector.
We show in the forthcoming paper \cite{KMOT3} that 
$M_0$ is an $H$-reflection with respect to the intersection form. 

We introduce a basis consisting of eigen components of
a reflection vector ${\bf v}$ of $M_0$.
Let ${\vcc}$ be the $J$-th component of ${\bf v}$ with respect to the 
simultaneous decomposition ${\bf v}=\sum_J{\vcc}$ associated with
\eqref{simaltaneous decomp int}.
We show in \cite{KMOT3} that 
$\{{\vcc}\}_J$ forms a basis of $\mathcal V$.

By restricting the intersection matrix
on $\mathcal V$ to $\mathcal V_j$ for each fixed $j\in \{1,\dots,p\}$, 
we have a $K$-skew intersection
form on the space $\mathcal V_j$. 
We show that $N_{0,j}$ is an $H$-reflection 
with respect to this intersection form, and
$$
{\bf v}_j=\sum_{J=(j_1,\dots, j_{m-1}, j)}{\vcc}
$$
is a reflection vector for $N_{0,j}$.
By this fact, we show that the intersection form on
$\mathcal V'_j$ is equal to that on $\mathcal V_j$  via the isomorphism
\eqref{twisted hom and infinitesimal inclusion}.

Using this coincidence,
we can determine the diagonal entries of 
the matrix $\mathcal H$ (up to constant) and the eigenvalue $\l$ of 
the circuit matrix $M_0$ by induction  on the number $m$ of the 
variables $x_1,\dots, x_m$.  
Hence the circuit matrix $M_0$ is expressed by   
an $H$-reflection in our main theorem.
As for a similar argument, see \cite[\S4]{Ma2}.

\begin{remark}
\begin{enumerate}
 \item 
 As for relations among $\rho_0,\rho_1,\dots,\rho_m\in \pi_1(X,\dot x)$,   
we focus on 
those in Theorem \ref{th:reduction}
and in Corollary \ref{circular relation}. 
They satisfy some more relations coming from singular points 
on the unirational hypersurface $S_R(x)=\{x\in \C^m\mid R(x)=0\}$,  
which is non singular only if $m=1$ or $(p,m)=(2,2)$.
\item
We study twisted homology groups and the intersection form
for the hypergeometric 
system $\cF_C^{p,m}(a,B)$ in \cite{KMOT3}.   
We give twisted cycles 
corresponding to the solutions $\Phi_J^\G(a,B;x)$
and the reflection vector of $M_0$.
\end{enumerate}
\end{remark}

\subsection{Notations}
The set $\{1,\dots, p\}(\subset \Z)$ is denoted by $[p]$. We set
$\bold e(z)=\exp(2\pi \sqrt{-1}z)$ and
$\z_p=\bold e(1/p)=\exp(2\pi\sqrt{-1}/p)$, which is a
primitive $p$-th root of unity.
Let $\Bbbk $ be a field. In this paper, $\Bbbk ^n$ and $\tr \Bbbk ^n$ denote the 
spaces of row and column vectors 
over $\Bbbk $. For a finite set $S$, $\Bbbk ^S$ denotes 
the set of row vectors ${\bf v}=(v_i)_{i\in S}$ indexed by $S$.
If $\# S=n$, then it is identified with $\Bbbk ^n$.
We similarly define a square matrix $M=(m_{i,j})_{i,j \in S}$
indexed by $S$.
For an element $j\in S$, the $j$-th element of the product ${\bf v}M$
is given as $\sum_{i\in S}v_im_{ij}$.
For a matrix $M$, $r(M)$ denotes 
the linear endomorphism on the space of row vectors obtained by
the right multiplication of $M$.
For a vector $(a_1,\dots, a_n)$, $\diag(a_1,\dots,a_n)$ denotes
the diagonal matrix whose diagonal elements are $a_1,\dots, a_n$.
For a finite set of subspaces $V_i$ in a vector space $V$,
$\sum_i V_i$ denotes the subspace consisting of linear 
combinations of elements in $V_i$,
and $\bigoplus_i V_i$ denotes the vector space consisting of
$(v_i)_i$ ($v_i \in V_i$).

\section{Fundamental group of the complement of the singular locus}

In this section, we study the fundamental group of the complement 
\begin{equation}
\label{eq:reg-locus}
X=\C^m-S(x)=\{x=(x_1,\dots,x_m)\in \C^m\mid x_1\dots x_mR(x)\ne 0\}
\end{equation} 
of the singular locus $S(x)$ of $\cF_C^{p,m}(a,B)$, where 
the set $S(x)$ and the polynomial $R(x)$ are given 
in Proposition \ref{result of part I}.
 
A path $\g$ in a topological space means a continuous map $t\mapsto \g(t)$ 
from a closed interval $[c_0,c_1]$ to this space, 
and $\g(c_0)$ and $\g(c_1)$ are called the start and end points of $\g$, 
respectively. 
The inverse path of $\gamma$ is $t\mapsto \g(c_0+c_1-t)$, which is  
denoted by $\gamma^{-1}$.
A path $\g$ satisfying  $\g(c_0)=\g(c_1)$ is called a loop, and 
the point $\g(c_0)(=\g(c_1))$ is called its base point. 
For two paths $\gamma$ and $\delta$ satisfying 
the end point of $\gamma$ is equal to the start point of $\delta$, 
$\gamma\cdot\delta$ denotes the path joining the start point of $\delta$ 
to the end point of $\gamma$.
If $\gamma$ and $\delta$ are loops with a common base point, 
then $\gamma\cdot \delta$  and $\delta\cdot \gamma$ are also loops.
For a loop  $\gamma$, its homotopy class in the fundamental group
is also denoted by $\gamma$.

\subsection{Covering of the complement of the singular locus $S(x)$}
\label{sec:cover}
We define a map 
\begin{equation}
\label{eq:cover-phi}
\varphi:\C^m_z\ni (z_1, \dots, z_m)\mapsto 
(x_1, \dots, x_m)=(z_1^p, \dots, z_m^p)\in \C^m_x. 
\end{equation}
Here $\C_z^m\simeq \C^m$ and $\C_x^m\simeq \C^m$
and their coordinates are given by $z=(z_1,\dots,z_m)$ and 
$x=(x_1,\dots,x_m)$.
The map $\f$ is of degree $p^m$ and ramified along 
$\{z\in \C_z^m\mid z_1\cdots z_m=0\}$. 
The restriction of this map to 
$(\C_z^\times)^m=\{z\in \C_z^m\mid z_1\cdots z_m\ne 0\}$
is also denoted by $\f$, and it is an unramified covering of 
$(\C_x^\times)^m=\{x\in \C_x^m\mid x_1\cdots x_m\ne 0\}$. 
Let $\sigma_k$ $(k=1,\dots,m)$ be automorphisms of $(\C_z^\times)^m$ given by 
\begin{equation*}
\sigma_k:(\C_z^\times)^m\ni (z_1,\dots,z_k,\dots,z_m)\mapsto 
(z_1,\dots,\z_pz_k,\dots,z_m)\in (\C_z^\times)^m.
\end{equation*}
Then any two of them commute each other, and  
each $\sigma_k$ of them is of order $p$ and satisfies $\f\circ \sigma_k=\f$. 
Thus the covering $\f:(\C_z^\times)^m\to (\C_x^\times)^m$ is normal, and 
$\sigma_1,\dots,\sigma_m$ generate its covering transformation group 
which is isomorphic to the additive group $(\Z/p\Z )^m$. 
By this isomorphism, 
the product of $\sigma_1^{i_1},\cdots, \sigma_m^{i_m}$
for $I=(i_1,\dots,i_m)\in (\Z/p\Z )^m$
is denoted by
\begin{equation}
\label{product of sigma}
\sigma_I=\sigma_1^{i_1}\cdots \sigma_m^{i_m}.  
\end{equation}

We consider the inverse image $Z=\f^{-1}(X)$ of $X$ in \eqref{eq:reg-locus} 
under the covering map $\f$.
We set 
$$S(z)=\f^{-1}(S(x)),\quad S_R(z)=\f^{-1}(S_R(x)),$$
which are the inverse image of the singular locus $S(x)$ of 
$\cF_C^{p,m}(a,B)$ and that of the hypersurface 
$$S_R(x)=\{x\in \C_x^m\mid R(x)=0\}.$$
It is easy to see that 
$$\f^{-1}(V(x_k))=V(z_k)$$
for $1\le k\le m$, where 
$$V(x_k)=\{x\in \C_x^m\mid x_k=0\},\quad V(z_k)=\{z\in \C_z^m\mid z_k=0\}$$
are coordinate hyperplanes.
Let 
$R(z)=R(z_1,\dots,z_m)$ be the pull back $\f^*(R(x))$ of 
the irreducible polynomial $R(x)$ in \eqref{eq:R(x)} under $\f$. 
Then 
it is decomposed into the product of $p^m$ linear forms 
\begin{equation*}
R(z_1,\dots,z_m)=\prod_{(i_1,\dots,i_m)\in (\Z/p\Z )^m}
(1-\z_p^{i_1}z_1-\cdots -\z_p^{i_m}z_m).
\end{equation*}
By using these linear forms, we define $p^m$ hyperplanes 
\begin{equation}
\label{eq:hyperplane}
H_I=\{(z_1, \dots, z_m)\in \C_z^m\mid 
\zeta^{i_1}_pz_1+\cdots +\zeta^{i_m}_pz_m=1\}
\end{equation} 
in $\C_z^m$ for $I=(i_1, \dots, i_m)\in (\Z/p\Z )^m$. 
Then we have  
$$
S_R(z)=\bigcup_{I\in (\Z/p\Z )^m}H_I,\quad 
S(z)=S_R(z)\cup \bigcup_{k=1}^m V(z_k).
$$
Since 
$$Z=\{z\in \C_z^m\mid \f^*(x_1\cdots x_mR(x))\ne 0\}=\{z\in \C_z^m\mid z_1\cdots z_mR(z)\ne 0\},$$ 
we have 
\begin{equation*}
Z=(\C_z^\times)^m-S_R(z)=
\C_z^m-\bigcup_{k=1}^m V(z_k)-
\bigcup_{I\in (\Z/p\Z )^m}H_I,
\end{equation*}
which is the complement of $m+p^m$ hyperplanes in $\C_z^m$. 


We choose a base point $\dot{z}=(\e,\dots,\e)\in Z$, 
where $\e$ is a sufficiently small positive number, 
and we set a base point $\dot{x}\in X$ by the image 
$\varphi(\dot{z})=(\e^p,\dots,\e^p)$ 
of $\dot z$ under $\f$. 
For a loop $\g$ with start point $\dot z$ in $Z$, 
the image $\f_*(\g)=\f \circ \g$ of $\g$ under $\f$ is a loop with start point 
$\dot x$ in $X$. Thus the map $\varphi$ induces a homomorphism
$$
\varphi_*:\pi_1(Z,\dot{z}) \to \pi_1(X,\dot{x}).
$$

\subsection{Definition of $\rho_0, \rho_1,\dots,\rho_m$ in the fundamental group}
\label{def of loop}
In this subsection, we define elements $\rho_0, \rho_1, \dots, \rho_m$
in $\pi_1(X,\dot{x})$.

\begin{definition}[Definition of $\rho_1, \dots, \rho_m$]
\label{def of rho}
We define a path $\widetilde{\rho_k}$ ($1\le k\le m$) connecting 
$\dot{z}$ 
and $\sigma_k\dot{z}=(\e,\dots,\overset{\text{$k$-th}}{\zeta_p}\e,\dots,
\e)$
given by 
\begin{equation*}
\widetilde{\rho_k}:[0,1]\ni t \mapsto 
(\e,\dots,\e,\overset{\text{$k$-th}}{\bold e(t/p)\e},\e,\dots,
\e)\in Z.
\end{equation*}
The image $\f_*(\widetilde{\rho_k})$ of $\widetilde{\rho_k}$ 
under the map $\f$ becomes a loop expressed by 
\begin{equation}
\label{eq:loops}
\rho_k:[0,1]\ni t \mapsto 
(\e^p,\dots,\e^p,\overset{\text{$k$-th}}{\bold e(t)\e^p},\e^p,\dots,
\e^p)\in X,
\end{equation}
and it defines an element of the fundamental group $\pi_1(X,\dot{x})$.

\end{definition}
Now we define $\widetilde{\rho_0} \in \pi_1(Z,\dot{z})$.
Let $L$ be the complex line passing through 
$\dot z$ and $(1/m,\dots, 1/m)\in H_I$ for $I=(0,\dots,0)$.
We choose a complex coordinate $s$ of $L$ so that 
$\dot{z}$ and $(1/m,\dots, 1/m)$ correspond to
$s=1$ and $s=0$, respectively.
Let $\delta$ be a small positive number and
$\dot{\beta}$ be the point in $L$ corresponding to 
$s=\delta$.

\begin{definition}[Definition of $\rho_0$]
\label{def of rho0}
We define a loop $\widetilde{\rho_0}$ in $L-S(z)$ 
by the composite of 
\begin{enumerate}
\item
the segment
from $\dot{z}$ to $\dot{\beta}$, 
\item
the small circle defined by
$$
[0,1]\ni u\mapsto s=\delta\bold e(u)\in L-S(z),
$$
and 
\item
the segment from $\dot{\beta}$ to $\dot{z}$.
\end{enumerate}

We define $\rho_0$ by 
 the image of $\widetilde{\rho_0}$ under the map $\f$, i.e., 
\begin{equation}
\label{eq:loop0}
\rho_0=\f_*(\widetilde{\rho_0})=\f\circ \widetilde{\rho_0},
\end{equation} 
which is a loop with start point $\dot x=(\e^p,\dots,\e^p)$ in $X$
turning once around the hypersurface $S_R(x)$ positively.
\end{definition}

\subsection{Infinitesimal inclusions and inductive structures}
\label{subsection:inftes incl}
We define in \S \ref{sec:cover} the affine space
$\C^m_x$ (resp. $\C^m_z$) and their closed subsets $S_R(x)$ and $S(x)$
(resp. $S_R(z)$ and $S(z)$). 
Recall that $V(z_m)$ is the affine subspace of $\C_z^m$ defined by
$$
V(z_m)=\{z=(z_1,\dots, z_{m-1}, z_{m})\in \C_z^m\mid z_m=0\},
$$
which is identified with
$$
\C_{z'}^{m-1}
=\{z'=(z_1,\dots, z_{m-1})\mid z_1,\dots, z_{m-1}\in \C\}.
$$ 
Under this identification, 
the restriction of the hyperplane arrangement $S_R(z)$ to $V(z_m)$ 
gives the hyperplane arrangement
$S_R(z')$ of $\C_{z'}^{m-1}$.
More precisely, we have
$$
S_R(z')=\bigcup_{I'\in (\Z/p\Z )^{m-1}} H'_{I'},
$$  
where
$$
H'_{I'}=\{(z_1, \dots,  z_{m-1})\subset \C_{z'}^{m-1}
\mid \zeta^{i_1}z_1+\cdots +\zeta^{i_{m-1}}z_{m-1}=1\}
$$
for $I'=(i_1,\dots, i_{m-1})\in (\Z/p\Z )^{m-1}$. 
We set 
$$
S(z')=S_R(z')\cup \{(z_1,\dots, z_{m-1})\mid z_1\cdots z_{m-1}=0\}.
$$
This construction is compatible with the action
of the covering transformation group of $\f:\C_z^m \to \C_x^m$.

The space $V(z_m)$ is stable under the action of $(\Z/p\Z )^{m}$
and its restriction induces an action of $(\Z/p\Z )^{m-1}$ on $V(z_m)$.
Its quotient is a subset of $\C_x^m$ defined by
\begin{equation}
\label{def of vxm}
V(x_m)=\{x=(x_1,\dots, x_{m-1}, x_{m})\in \C_x^m\mid x_m=0\},
\end{equation} 
which is identified with the space 
$$
\C_{x'}^{m-1}=\{x'=(x_1,\dots, x_{m-1})\mid x_1,\dots, x_{m-1}\in \C\},
$$
and the quotient map $\C_{z'}^{m-1}\to \C_{x'}^{m-1}$ is denoted by 
$\f'$.
The natural inclusions $S_R(z')\to \C_{z'}^{m-1}$ and $S(z')\to \C_{z'}^{m-1}$ 
also induce 
inclusions $S_R(x')\to \C_{x'}^{m-1}$ and $S_R(x')\to \C_{x'}^{m-1}$, 
where $S_R(x')$ and $S(x')$ are the quotients of $S_R(z')$ and $S(z')$
under the action of $(\Z/p\Z )^{m-1}$, respectively.
We set
\begin{equation}
\label{one dim lower}
X'=\C_{x'}^{m-1}-S(x'),\quad Z'=\C_{z'}^{m-1}-S(z').
\end{equation}
This inductive structure is used to study the restriction 
of the space of solutions to the system 
$\cF_C^{p,m}(a,B)$ around the neighborhood of $x_m=0$
in the following section. 

In this subsection, 
we study the fundamental groups of the complements $X$ and
$Z$ of the closed subsets $S(x)$ and $S(z)$ using this inductive structure.
We prepare several notations.
Let $U_Z$ be a tubular neighborhood
of $V(z_m)\subset \C^m_z$.
Let $I'$ be an element of $(\Z/p\Z )^{m-1}$
and $U_{I'}$ be a tubular neighborhood of $H'_{I'}$ in $\C_{z'}^{m-1}$,
and set 
$$U_R(z')=\bigcup_{I'\in (\Z/p\Z )^{m-1}}U_{I'}.$$
Then $U_R(z')$ is a tubular neighborhood
of $S_R(z')\subset \C^{m-1}_{z'}$.
We assume that 
$Z'-U_R(z')$
is a deformation retract of $Z'$.
We also assume that $U_Z$ and $U_R(z')$ are stable under the actions of
$(\Z/p\Z )^m$ and $(\Z/p\Z )^{m-1}$.
We set
$$
\varphi(U_Z)=U_X,\quad \varphi'(U_R(z'))=U_R(x').
$$
We use the covering transformation 
$$
\sigma_m:Z\ni (z_1,\dots,z_m)\mapsto (z_1, \dots, \zeta_p z_m)\in Z.
$$
Let
$$
\pi_m:U_Z \to V(z_m),\quad \pi_m^\circ:U_Z^{\circ} \to V(z_m)
$$
be the restriction 
of the projection $\C_{z}^{m}\to\C_{z'}^{m-1}$
to $U_Z$ and that to $U_Z^\circ=U_Z-V(z_m)$.
Then they are a disc bundle and a punctured disc bundle.
Moreover we also assume that $S_R(z)\cap U_Z \subset \pi^{-1}_m(U_R(z'))$.
\begin{figure}[h]

\includegraphics[width=8cm]{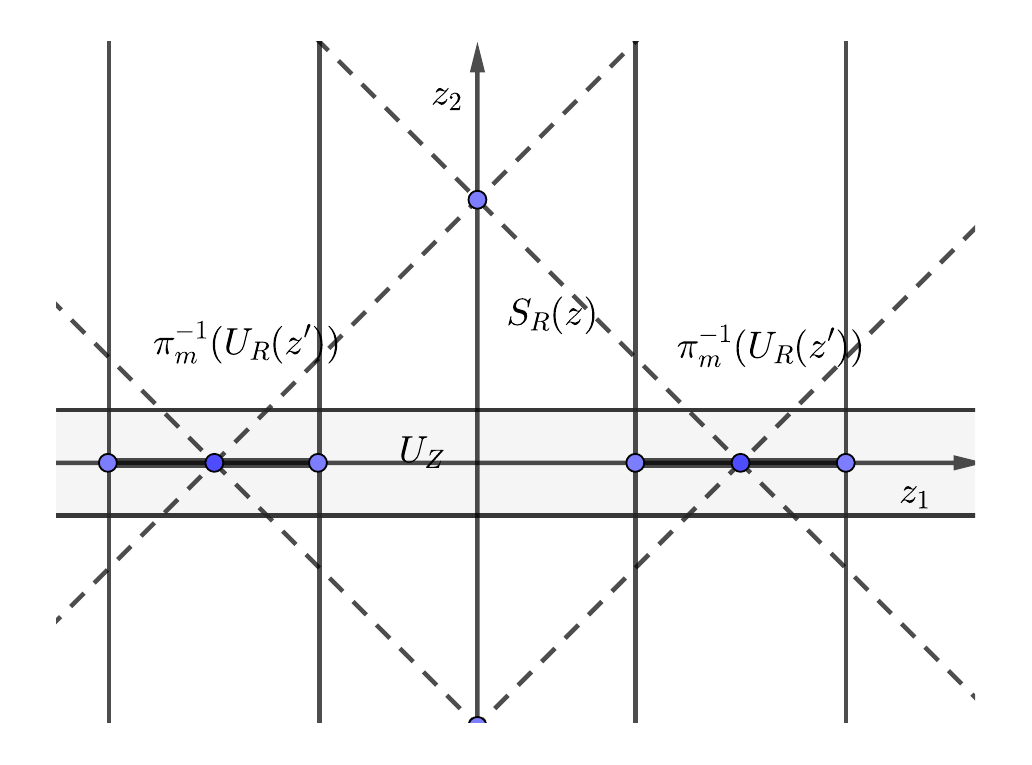}

\caption{Tubular neighborhood}
\end{figure}
Then the restriction 
$$
(\pi_m^{\circ})^{-1}(Z'-U_R(z'))
\to Z'-U_R(z') 
$$
of 
$\pi_m^\circ$ to $(\pi_m^{\circ})^{-1}(Z'-U_R(z'))$
is a punctured disc bundle, and we have 
${\pi_m^{\circ}}^{-1}(Z'-U_R(z'))\subset Z$.
We define a set $\C_{z',\e}^{m-1}$ and a morphism $\widetilde{r}$ by
\begin{align*}
&\C_{z',\e}^{m-1}=\{z=(z_1,\dots, z_m)\in\C_z^{m}\mid z_m=\e\},
\\
&\widetilde{r}:\C_{z',\e}^{m-1}\ni(z_1, \dots, z_{m-1},\e)
\mapsto (z_1, 
\dots, z_{m-1})\in \C_{z'}^{m-1}.
\end{align*}
Then the fundamental group of the subset
$$
Z'_{\e}=\C_{z',\e}^{m-1}-{\pi_m^{\circ}}^{-1}(U_R(z'))
$$
is isomorphic to
that of $Z'-U_R(z')$ via the map $\widetilde{r}$.
We set $\dot{z}'=\pi_m^{\circ}(\dot{z})$.
Thus we have the following homomorphisms of fundamental groups:
$$
\xymatrix{
 & \pi_1(Z'_{\e},\dot{z}) \ar[ld]_{\simeq}\ar[rd]\ar[d]&
\\
\pi_1(Z'-U_R(z'),\dot{z}')\ar[d]_{\simeq}
&\pi_1({\pi_m^\circ}^{-1}(Z'-U_R(z')),\dot{z})\ar[l]\ar[r] &
\pi_1(Z,\dot{z})
\\
 \pi_1(Z',\dot{z}').
}
$$
As a consequence, we have a homomorphism:
\begin{equation}
\label{inif incl Z}
\iota_{\infty}:\pi_1(Z',\dot{z}') \to \pi_1(Z,\dot{z}),
\end{equation}
which is called the infinitesimal embedding.
This construction is compatible with the actions of $(\Z/p\Z )^m$ 
and $(\Z/p\Z )^{m-1}$. The image of $Z'_{\e}$ under
the quotient map is denoted by  
$X'_{\e}$.
Then we have the infinitesimal inclusion
from $X'_{\e}$ to $X$ and a homomorphism:
\begin{equation}
\label{inif incl X}
\iota_{\infty}:
\pi_1(X',\dot{x}') \to \pi_1(X,\dot{x}). 
\end{equation}

Let $\widetilde{\rho_0'}$ and $\rho_0'$ be the elements of 
$\pi_1(Z',\dot{z}')$ and $\pi_1(X',\dot{x}')$ constructed in 
Definition \ref{def of rho0} 
as the case of $m-1$.
An element $\gamma$ in $\pi_1(Z',\dot{z}')$ can be represented by
a loop in $Z'-U_R(z')$ expressed as
$$
[0,1]\ni t \mapsto \gamma(t)=(z_1(t),\dots, z_{m-1}(t)) \in Z'.
$$
For this expression, the element $\iota_{\infty}(\gamma)$ in
$\pi_1(Z,\dot{z})$
is expressed as
$$
[0,1]\ni t \mapsto \zeta(t)=(z_1(t),\dots, z_{m-1}(t),\e) \in Z
$$
for a sufficiently small $\e$.
\begin{theorem}
\label{th:reduction}
\begin{enumerate}
 \item 
In the fundamental group $\pi_1(X,\dot{x})$, we have an identity
\begin{equation}
\label{conj prod in X}
\iota_{\infty}(\rho_0')=\rho_0\cdot (\rho_{m}\cdot \rho_0\cdot \rho_m^{-1})\cdot
(\rho_m^2\cdot \rho_0\cdot\rho_m^{-2})
\cdots 
(\rho_m^{p-1}\cdot \rho_0\cdot\rho_m^{1-p}).
\end{equation}
\item
The element $\iota_{\infty}(\rho_0')$ commutes with $\rho_m$.
\end{enumerate}
\end{theorem}

\subsection{Proof of Theorem \ref{th:reduction}}
We prepare several lemmas.
\begin{lemma}
\label{triangle no intersection}
Let $\Delta_2$ be a $2$-dimensional 
simplex with the vertices
$$
p_0=(\e,\dots, \e),\quad 
p_1=(\frac{1}{m},\dots, \frac{1}{m}),\quad p_2= (\frac{1-\e}{m-1},
\dots, \frac{1-\e}{m-1},\e)
$$
in $\C_z^m$.
If a point $p$ is contained in $\Delta_2\cap H_I$,
then 
$p$ is on the edge with the vertices
$p_1,p_2$ and $I=(0,\dots, 0)$.
In particular, the interior $\Delta_2^\circ$ 
of $\Delta_2$ does not meet $S_R(z)$.
\end{lemma}

\begin{figure}[h]
\includegraphics[width=7cm]{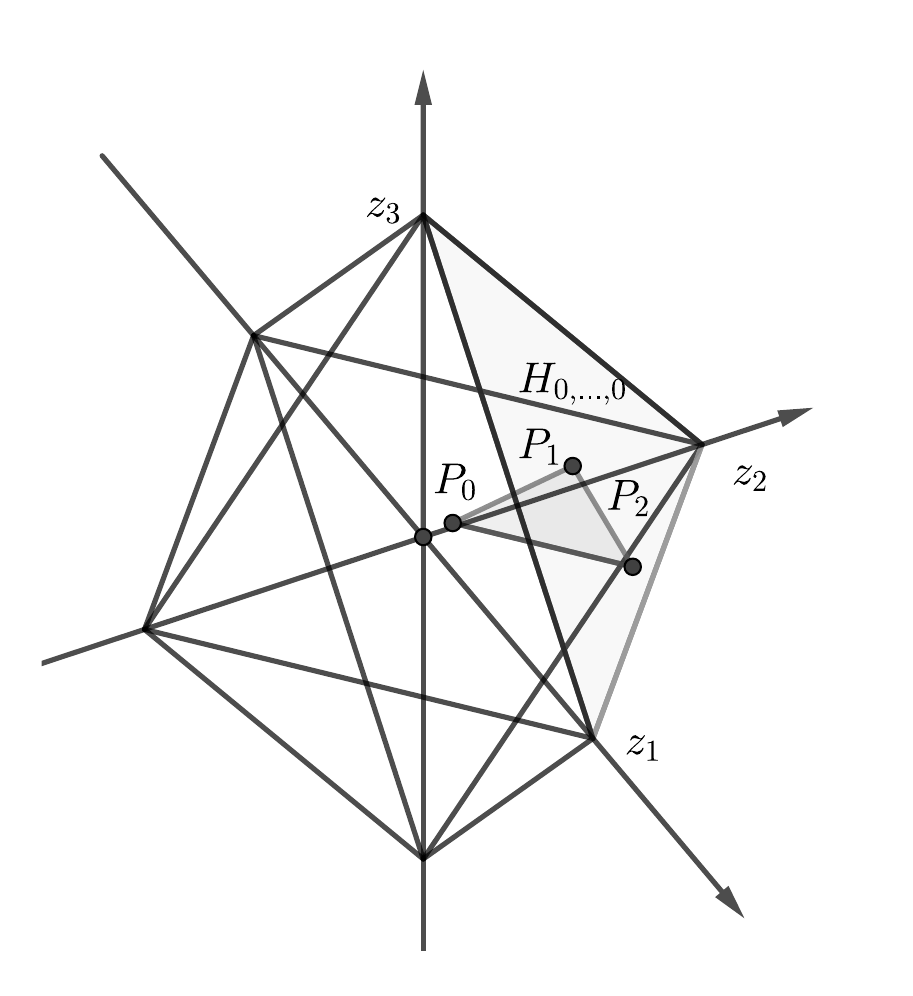}

\caption{Changing paths up to homotopy}
\end{figure}

\begin{proof}
Let $p$ be a point in $\Delta_2\cap H_I$
for $I=(i_1, \dots, i_m)$. 
Then it can be written as 
$$
p=(r, \dots,r, rs),\quad (0\leq r(m-1+s)\leq 1, 0\leq s \leq 1,\e \leq rs).
$$
Since $p$ is in $H_I$, we have
$$
\zeta^{i_1}r+\cdots+\zeta^{i_{m-1}}r+\zeta^{i_{m}}sr=1, \quad (r>0).
$$
By triangle inequality, we have
\begin{align*}
r(m-1+s)\leq &1=\mid\zeta^{i_1}r+\cdots+\zeta^{i_{m-1}}r+\zeta^{i_{m}}sr\mid  
\\
\leq &
\mid r\mid +\cdots+\mid r\mid+\mid sr \mid =r(m-1+s).
\end{align*}
Therefore the above inequalities  
should be equalities.
Since $r>0, rs>0$, we have $i_1=\cdots=i_m=0$.
\end{proof}
To prove Theorem \ref{th:reduction}, we define paths 
$\rho_0^*$ and $\rho_{m,k}^*$ in $Z$.
Let $L'$ be a complex line passing through $(0,\dots, 0)$ and 
$(1/(m-1),\dots, 1/(m-1))$
in $\C_{z'}^{m-1}$, and $L'_{\e}$ be a line in $\C_{z',\e}^{m-1}$
given by $\{(z',\e)\mid z'\in L'\}$.
For an element $(z_1, \dots, z_{m-1},z_m)\in {\pi_m}^{-1}(L')$, 
we introduce 
coordinates 
$(s,z_m)$ 
so that 
\begin{align*}
&(z_1, \dots, z_{m-1},z_m)=(\dfrac{1-s}{m-1},\dots, \dfrac{1-s}{m-1},z_m).
\end{align*}
Then $L'_\e$ is defined by $z_m=\e$ and
the point
$\dot{z}=(\e,\dots, \e,\e) \in L'_\e$ is given by 
$(s,z_m)=(1-(m-1)\e,\e)$.
Under these coordinates, the action of $\sigma_m$ is given by
$$
\sigma_m(s,z_m)=(s,\zeta_p z_m).
$$
The intersections
$L'_{\e}\cap H_{0,\dots, 0}$ and
$\sigma_m^k(L'_{\e}\cap H_{0,\dots, 0})$
are expressed by $(s,z_m)=(\e,\e)$ and 
$(s,z_m)=(\e,\bold e(k/p)\e)$, respectively.
Therefore we have
$$
\sigma_m^k(L'_{\e}\cap H_{0,\dots, 0})\in H_{0,\dots,0,-k},
$$
and the point $L'_{\e}\cap H_{0,\dots,0,-k}$ 
is given by $(s,z_m)=(\bold e(-k/p)\e,\e)$.
Let $\delta$ be a positive real number
sufficiently smaller than $\e$, and
the point $(s,z_m)=(\e+\delta,\e)\in L'_\e$ 
be denoted by $b_L$.
Then we have
$$
\sigma_m^k(b_L)=(\e+\delta,\bold e(k/p)\e)\in \sigma_m^k(L'_{\e}).
$$

\begin{definition}
\begin{enumerate}
\item
We define paths and a loop as follows:
\begin{align*}
&\text{the segment }\phi\text{ from $\dot{z}$ to $b_L$ in $L'_{\e}$},
\\
&\text{the loop }\rho_0^{\dag}:
[0,1]\ni t\mapsto (\e+\delta\bold e(t),\e)\in L'_{\e},
\\
&\text{the path }
\rho_{m,k}^{\dag}:
[0,1]\ni t\mapsto (\e+\delta,\bold e(kt/p)\e)\in Z,
\\
&\text{the path }
\widetilde{\rho_{m,k}}:
[0,1]\ni t\mapsto (1-(m-1)\e,\bold e(kt/p)\e)\in Z, 
\end{align*}
where $k=0,1,\dots,p-1$. 
\item
We define a 
loop ${\rho_0'}^{\dag}$ in
$L'_{\e}$ by
$$
[0,1]\ni t\mapsto (s,z_m)=((\e+\delta)\bold e(t),\e).
$$
Then we have
$$
\iota_{\infty}(\widetilde{\rho_0'})=\phi\cdot\rho_0^{'\dag}\cdot\phi^{-1}.
$$
\item
We set $\rho_0^{*}=\phi\cdot\rho_0^{\dag}\cdot\phi^{-1}$,
$\rho_{m,k}^{*}=\phi\cdot\rho_{m,k}^{\dag}\cdot\sigma_m^k(\phi)^{-1}$.
\end{enumerate} 
\end{definition}

\begin{figure}[h]

\includegraphics[width=10cm]{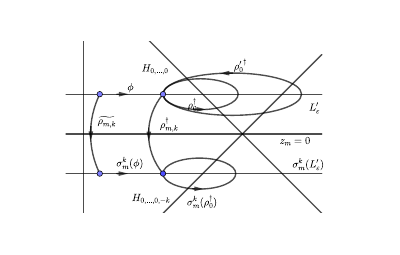}

\caption{Local changing of paths}
\end{figure}

We have the following lemma.
\begin{lemma}
\label{homo eq tubler}
The loop $\rho_0^*$ is homotopy equivalent to $\widetilde{\rho_0}$ in $Z$.
The path $\rho_{m,k}^*$ is homotopy equivalent to 
$\widetilde{\rho_{m,k}}$ in $Z$.
\end{lemma}
\begin{proof}
The statement for $\rho_0^*$ is a direct consequence of 
Lemma \ref{triangle no intersection}.
It is easy to see that the paths $\widetilde{\rho_{m,k}}$ and $\rho_{m,k}^*$ 
are homotopy equivalent to each other in $Z$. 
\end{proof}
Under the group homomorphism 
$\f_*:\pi_1(Z,\dot{z})\to \pi_1(X,\dot{x})$, we have
\begin{align*}
\varphi_*(
\widetilde{\rho_{m,k}}\cdot \sigma_m^k(\widetilde{\rho_0})
\cdot 
\widetilde{\rho_{m,k}}^{-1}
)=
\rho_{m}^k\cdot \rho_0\cdot \rho_m^{-k}.
\end{align*}
Since the loop
$\widetilde{\rho_{m,k}}\cdot \sigma_m^k(\widetilde{\rho_0})
\cdot 
\widetilde{\rho_{m,k}}^{-1}$ 
is homotopy equivalent to
$\rho^*_{m,k}\cdot \sigma_m^k(\rho^*_0)\cdot 
{\rho^*_{m,k}}^{-1}$ 
in $Z$ by Lemma \ref{homo eq tubler},
to prove \eqref{conj prod in X},
it is enough to show the identity
\begin{align}
\label{neighborhood of M rel}
\iota_{\infty}(\widetilde{\rho_0'})=
&\rho^*_0\cdot 
(\rho^*_{m,1}\cdot \sigma_m(\rho^*_0)\cdot {\rho^*_{m,1}}^{-1})
\cdot
({\rho^*_{m,2}}\cdot \sigma_m^2({\rho_0^*})\cdot {\rho_{m,2}^*}^{-1})
\\
\nonumber
&\cdots 
(\rho_{m,p-1}^*\cdot \sigma_m^{p-1}(\rho_0^*)\cdot {\rho_{m,p-1}^*}^{-1})
\end{align}
in $\pi_1(U_Z-S(z),\dot{z})$.
Since 
$$
{\rho^*_{m,k}}\cdot \sigma_m^k({\rho_0^*})\cdot {\rho_{m,k}^*}^{-1}
=\phi{\rho^{\dag}_{m,k}}\cdot \sigma_m^k({\rho_0^{\dag}})\cdot 
{\rho_{m,k}^{\dag}}^{-1}
\phi^{-1},
$$
\eqref{neighborhood of M rel} is equivalent to 
the equality
\begin{align}
\label{neighborhood of M rel**}
{\rho_0'}^{\dag}=
&\rho^{\dag}_0\cdot 
(\rho^{\dag}_{m,1}\cdot \sigma_m(\rho^{\dag}_0)\cdot {\rho^{\dag}_{m,1}}^{-1})
\cdot
({\rho^{\dag}_{m,2}}\cdot \sigma_m^2({\rho_0^{\dag}})\cdot {\rho_{m,2}^{\dag}}^{-1})
\\
\nonumber
&\cdots 
(\rho_{m,p-1}^{\dag}\cdot \sigma_m^{p-1}(\rho_0^{\dag})\cdot {\rho_{m,p-1}^{\dag}}^{-1}).
\end{align}
For a sufficiently small $\e$ (for example, $\e<\frac{\sin (\pi/p)}{2(m-1)}$), we define the subsets $B_{\loc}, S(z)_{\loc}$ of
${\pi_m}^{-1}(L')$ by
\begin{align*}
&B_{\loc}=\{(s,z_m)\in {\pi_m}^{-1}(L') \mid |s|<  2\e, |z_m|< 2\e\},
\\
& S(z)_{\loc}=B_{\loc}\cap S(z)=
\{(s,z_m)\in B_{\loc} \mid z_m(s^p-z_m^p)= 0\},
\end{align*}
and show the equality \eqref{neighborhood of M rel**}
in $\pi_1(B_{\loc}-S(z)_{\loc},\dot{z})$.
In the space $B_{\loc}$, we have 
\begin{align*}
&B_{\loc}\cap H_{0,\dots,0,-k}=
\{(s,z_m)\in B_{\loc} \mid s=\bold e(-k/p)z_m\},
\end{align*}
and the paths
${\rho_0'}^{\dag},
\rho^{\dag}_0$ and
$\rho^{\dag}_{m,k}$
are contained in $B_{\loc}-S(z)_{\loc}$.
Then $B_{\loc}$ is stable under the action of 
$\sigma_m$.
We define a path $A_k$ in $B_{\loc}-S(z)_{\loc}$ by
$$
[0,1]\ni t \mapsto (s,z_m)=((\e+\delta)\bold e(-kt/p),\e)
$$
for $k=0,1,\dots,p-1$, and an action $\sigma_m'$ on $B_{\loc}$ by
$$
\sigma'_m:(s,z_m)\mapsto (\zeta_p s,z_m).
$$
Under these notations, we have the following lemma.
\begin{lemma}
\label{local lemma}
The loop
$\rho_{m,k}^{\dag}\cdot\sigma_m^k(\rho_0^{\dag})\cdot (\rho_{m,k}^{\dag})^{-1}$ 
is homotopy equivalent to
$$
\gamma_k=A_{k}\cdot {\sigma_m'}^{-k}(\rho_0^{\dag})\cdot A_{k}^{-1}
$$
in $B_{\loc}-S(z)_{\loc}$. The homotopy class of the loop $\gamma_k$ is an element in
$\pi_1(L'_{\e}\cap (B_{\loc}-S(z)_{\loc}),b_L)$.
As a consequence,
$\widetilde{\rho_{m,k}}\cdot\sigma_m^k(\widetilde{\rho_0})
\cdot\widetilde{\rho_{m,k}}^{-1}$
is  homotopy equivalent to $\phi \cdot \gamma_k \cdot \phi^{-1}$. 
See Figure \ref{Paths} for the paths ${\rho_0'}^{\dag}$ and $\gamma_k$.
\end{lemma}
\begin{figure}[h]
\includegraphics[width=8.0cm]{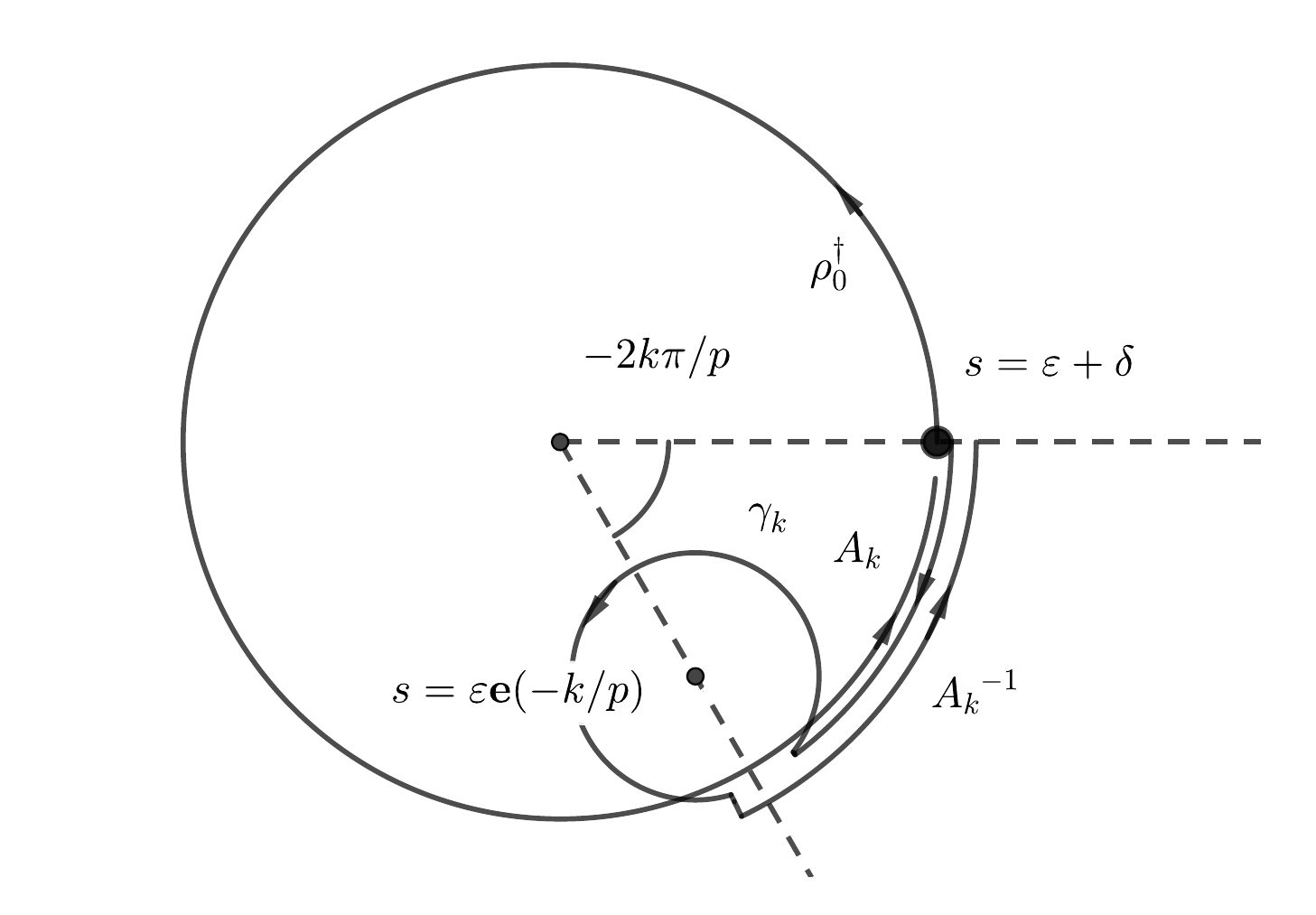}
\caption{The paths $\rho_0^{\dag}$ and $\gamma_k$}
\label{Paths}
\end{figure}

\begin{proof}
We explicitly construct a homotopy between
$\rho_{m,k}^{\dag}\cdot\sigma_m^k(\rho_0^{\dag})\cdot (\rho_{m,k}^{\dag})^{-1}$ 
and $\gamma_k$ in $B_{\loc}$ parameterized by $0\leq \theta\leq k/p$.
We fix $\theta$.
Let $L'_{\e,\theta}$ be the subset of $B_{\loc}$ defined by
$z_m=\bold e(\theta)\e$. Then the intersection 
$L'_{\e,\theta}\cap H_{0,\dots,0,-k}$ is 
$(s,z_m)=(\bold e(\theta-k/p)\e,\bold e(\theta)\e)$.
For $0\leq \theta\leq k/p$, we define a loop $\gamma_{k,\theta}$
in $B_{\loc}-S(z)_{\loc}$ by
$$
\gamma_{k,\theta}=\rho_{m,\theta}^{\dag}\cdot A_{k,\theta}
\cdot\rho^{\dag}_{0,\theta}\cdot A_{k,\theta}^{-1}\cdot
(\rho_{m,\theta}^{\dag})^{-1},
$$
where
\begin{enumerate}
 \item 
a path $\rho_{m,\theta}^{\dag}$ connecting
$(\delta+\e,\e)$ and $(\delta+\e,\bold e(\theta)\e)$:
$$
[0,1]\ni t \mapsto (\delta+\e ,\bold e(t\theta)\e),
$$
\item
a path $A_{k,\theta}$ connecting
$(\delta+\e,\bold e(\theta)\e)$ and
$((\delta+\e)\bold e(\theta-k/p),\bold e(\theta)\e)$:
$$
[0,1]\ni t \mapsto (s,z_m)=((\delta+\e)\bold e(t(\theta-k/p)),
\bold e(\theta)\e),
$$
\item
and a circle $\rho^{\dag}_{0,\theta}$ with  center $(\bold e(\theta-k/p)\e,\bold e(\theta)\e)$:
$$
[0,1]\ni t \mapsto (s,z_m)=((\e+\delta\bold e(t))\bold e(\theta-k/p)),
\bold e(\theta)\e).
$$
\end{enumerate}
Then we have
$$
\gamma_{k,0}=\gamma_k,\quad 
\gamma_{k,k/p}=\rho_{m,k}^{\dag}\cdot\sigma_m^k(\rho_0^{\dag})\cdot (\rho_{m,k}^{\dag})^{-1},
$$
which mean the homotopy equivalence of $\gamma_{k,0}$
and $\gamma_{k,k/p}$.
\end{proof}
\begin{proof}[Proof of Theorem \ref{th:reduction}]
(1) 
By Lemma \ref{local lemma}, we have the equality
\begin{align}
\label{local equality in infinitesimal line}
\gamma_0\gamma_1\cdots \gamma_{p-1}=
&\rho^{\dag}_0\cdot 
(\rho^{\dag}_{m,1}\cdot \sigma_m(\rho^{\dag}_0)\cdot {\rho^{\dag}_{m,1}}^{-1})
\cdot
({\rho^{\dag}_{m,2}}\cdot \sigma_m^2({\rho_0^{\dag}})\cdot {\rho_{m,2}^{\dag}}^{-1})
\\
\nonumber
&\cdots 
(\rho_{m,p-1}^{\dag}\cdot \sigma_m^{p-1}(\rho_0^{\dag})\cdot {\rho_{m,p-1}^{\dag}}^{-1})
\end{align}
in $\pi_1(B_{\loc}-S(z)_{\loc},b_L)$.
Since the loop $\gamma_0\gamma_1\cdots \gamma_{p-1}$
and the loop ${\rho_0'}^{\dag}$ are represented 
by loops in the complex one-dimensional subspace 
$L'_{\e}\cap (B_{\loc}-S(z)_{\loc})$ 
surrounding $(\bold e(-k/p),\e)$, $k=0,\dots, p-1$.
We 
easily see that these loops are homotopy equivalent
to each other in this subspace.
Thus we have proved the equality 
\eqref{neighborhood of M rel**}.

\smallskip\noindent
(2)\ 
The loop
$$
\rho_m\cdot\big(\rho_0\cdot (\rho_{m}\cdot \rho_0\cdot \rho_m^{-1})\cdot
(\rho_m^2\cdot \rho_0\cdot\rho_m^{-2})
\cdots 
(\rho_m^{p-1}\cdot \rho_0\cdot\rho_m^{1-p})\big)\cdot\rho_m^{-1}
$$
in $X$ is lifted to a loop 
$$
 (\rho^*_{m}\cdot \sigma_m(\rho^*_0)\cdot {\rho^*_m}^{-1})\cdot
({\rho^*_m}^2\cdot \sigma_m^2({\rho_0^*})\cdot ({\rho_m^*}^2)^{-1})
\cdots 
({\rho_m^*}^{p}\cdot \sigma_m^{p}(\rho_0^*)\cdot ({\rho_m^*}^{p})^{-1})
$$
in $Z$. By Lemma \ref{local lemma} again, the above element is  homotopy 
equivalent to $\gamma_1\gamma_2\cdots \gamma_{p}$.
We can show  that it is homotopy equivalent to ${\rho_0'}^{\dag}$ which is
equal to $\gamma_0\cdots \gamma_{p-1}$
by \eqref{neighborhood of M rel**}
in
$L'_{\e}\cap (B_{\loc}-S(z)_{\loc})$.
Therefore we have a relation 
\begin{align*}
&\rho_m\cdot\big(\rho_0\cdot (\rho_{m}\cdot \rho_0\cdot \rho_m^{-1})\cdot
(\rho_m^2\cdot \rho_0\cdot\rho_m^{-2})
\cdots 
(\rho_m^{p-1}\cdot \rho_0\cdot\rho_m^{1-p})\big)
\\
=
&\big(\rho_0\cdot (\rho_{m}\cdot \rho_0\cdot \rho_m^{-1})\cdot
(\rho_m^2\cdot \rho_0\cdot\rho_m^{-2})
\cdots 
(\rho_m^{p-1}\cdot \rho_0\cdot\rho_m^{1-p})\big)\cdot\rho_m 
\end{align*}
as elements of $\pi_1(X,\dot{x})$.
\end{proof}

\begin{cor}
\label{circular relation}
The elements $\rho_0,\rho_1, \dots, \rho_m$ in $\pi_1(X,\dot{x})$
satisfy the following relations
\begin{equation}
\label{eq:pi1-relations}
\rho_k\cdot \rho_{k'}=\rho_{k'}\cdot \rho_k,
\quad (\rho_0\cdot \rho_k)^p=(\rho_k\cdot \rho_0)^p
\end{equation}
for any $k,k'\in \{1,\dots,m\}$.  
\end{cor}
\begin{proof}
For a sufficiently small neighborhood $U_O$ of the origin $O$,
$\pi_1(U_O-S(x),\dot{x})$ is commutative and
the elements $\rho_1, \dots, \rho_m$ are in the image of 
$\pi_1(U_O-S(x),\dot{x})$. Thus we have the first commutativity.

By Theorem \ref{th:reduction} (1),
the element
$
\iota_{\infty}(\rho_0')=(\rho_0\cdot \rho_{m})^p\cdot ({\rho_{m}})^{-p}
$
commutes with $\rho_{m}$. Thus $(\rho_0\cdot \rho_{m})^p$ also commutes with 
$\rho_{m}$. Hence, we have the second relation.
Since Proof of Theorem \ref{th:reduction} (1),(2) is symmetric for 
$\rho_1,\dots, \rho_m$,
the relation holds for arbitrary $k$ ($1\leq k \leq m$).
\end{proof}

\subsection{Generators of $\pi_1(X,\dot{x})$}
Let $G$ be a group, and $g_1, \dots, g_k$ be elements in $G$.
The subgroup generated by $g_1, \dots, g_k$ is denoted by
$\langle g_1, \dots, g_k\rangle.$
In this subsection, we prove the following theorem.
\begin{theorem}
\label{generator of fund of X}
Let $\rho_0,\rho_1, \dots, \rho_m$ be the elements
in $\pi_1(X,\dot x)$ defined in Definitions \ref{def of rho} and
\ref{def of rho0}.
Then we have
$$
\pi_1(X,\dot x)=\langle \rho_0, \rho_1,\dots,\rho_m \rangle.
$$
\end{theorem}

For $I=(i_1, \dots, i_m) \in (\Z/p\Z )^m$, we have
the automorphism $\sigma_I$
of $Z$ defined in \eqref{product of sigma},
and the path $\widetilde{\rho_{I}}$ in $Z$ from $\dot{z}$ to 
$\sigma_I(\dot{z})$ by
$$
 \widetilde{\rho_{I}}:[0,1]\ni t \mapsto (\bold e(i_1t/p)\e,\dots,\bold e(i_mt/p)\e)\in Z,
 $$
where each representative $i_k$ $(1\le k\le m)$ of $\Z/p\Z$ is chosen 
from the set $\{0,1,\dots,p-1\}$.  Then $\widetilde{\rho_{I}}$ is a lifting of 
 $\rho_1^{i_1}\cdots \rho_m^{i_m}\in \pi_1(X,\dot{x})$ 
for the covering $\varphi:Z\to X$.
For this $I=(i_1, \dots, i_m)$, we set $I'=(i_1,\dots, i_{m-1},0)$.
Then the path 
$$
\widetilde{\rho_{I'}}\cdot
\sigma_{I'}(\phi \cdot \gamma_{i_m} \cdot \phi^{-1})
\cdot 
\widetilde{\rho_{I'}}^{-1}
$$
in $\pi_1(Z\cap \C_{z',\e}^{m-1},\dot{z})$ is  homotopy equivalent 
to
$$
\Ad_I(\widetilde{\rho_0})=
\widetilde{\rho_{I}}\cdot
\sigma_I(\widetilde{\rho_0})\cdot (
\widetilde{\rho_{I}})^{-1}
$$
as a loop in $Z$ 
by Lemma \ref{local lemma}.
Then we have
 $$
 \f_*(\Ad_I(\widetilde{\rho_0}))=
 (\rho_1^{i_1}\cdots \rho_m^{i_m})\rho_0 (\rho_1^{i_1}\cdots \rho_m^{i_m})^{-1}.
 $$
We define the loop $\widehat{\rho_k}$ in $Z$
by
$$
[0,1] \ni t \mapsto
(\e,\dots,\e,\overset{\text{$k$-th}}{\bold e(t)\e},\e,\dots,
\e)\in Z,
$$
which is a lifting of $\rho_k^p\in \pi_1(X,\dot{x})$.
To prove Theorem \ref{generator of fund of X}, we give the following theorem.

\begin{theorem}
\label{generator for Z}
Under the above notation, we have
\begin{equation}
\label{st gen}
\pi_1(Z,\dot{z})=
\langle \widehat{\rho_1},\dots,\widehat{\rho_m},
\Ad_I(\widetilde{\rho_0})\rangle
_{I\in (\Z/p\Z )^m}.
\end{equation}
\end{theorem}
\begin{proof}
We prove the proposition by induction on  $m$.
For $m=1$, the proposition holds obviously.
We assume 
that the proposition holds for $m-1$ ($m\geq 2$). 
Let $\widehat{\rho_1}', \dots, \widehat{\rho_{m-1}}'$
and $\Ad_{I'}(\widetilde{\rho_0'})$ ($I'\in (\Z/p\Z )^{m-1}$)
be loops in $Z'=\C_{z'}^{m-1}-S(z')$ constructed as the case of $m-1$. 
In \S \ref{subsection:inftes incl}, we see  
that the isomorphism $\widetilde{r}$ induces an isomorphism
$\pi_1(Z',\dot{z'})\xrightarrow{\iota_{\infty}}
\pi_1(Z'_{\e},\dot{z})$.
By the assumption of the induction, we have
\begin{align*}
\pi_1(Z'_{\e},\dot{z})
&=\langle \iota_{\infty}(\widehat{\rho_1}'), 
\dots, \iota_{\infty}(\widehat{\rho_{m-1}}'),
\iota_{\infty}(\Ad_{I'}(\widetilde{\rho_0'}))\rangle_ {I'\in (\Z/p\Z )^{m-1}}
\\
&=\langle \widehat{\rho_1}, 
\dots, \widehat{\rho_{m-1}},
\iota_{\infty}(\Ad_{I'}(\widetilde{\rho_0'}))
\rangle_ {I'\in (\Z/p\Z )^{m-1}}.
\end{align*}
Since the hyperplane $\C_{z',\e}^{m-1}$ 
intersects the components of $S(z)$ except 
the hyperplane $V(z_m)$,
we have
$$
\pi_1(Z,\dot{z})=
\langle\pi_1(Z\cap \C_{z',\e}^{m-1},\dot{z}),
\widehat{\rho_m} \rangle.
$$


For an element $I'\in (\Z/p\Z )^{m-1}$, the subgroup of 
$\pi_1(Z\cap \C_{z',\e}^{m-1},\dot{z})$ 
generated by 
$\{\Ad_I(\widetilde{\rho_0})\}_{\varpi_m(I)=I'}$
is denoted by $\Pi_{I'}$, where
$\varpi_m$ is a map ${(\Z/p\Z)^{m}}\ni(i_1, \dots, i_m) 
\mapsto (i_1, \dots, i_{m-1})\in(\Z/p\Z)^{m-1}$.
We consider the union 
$$
(Z\cap \C_{z',\e}^{m-1})-Z_{\e}'=
\bigcup_{I'\in (\Z/p\Z )^{m-1}} 
(\pi_m^{-1}(U_{I'})\cap \C_{z',\e}^{m-1}\cap Z),
$$
and the inclusion
$$
\bigg(\big[\pi_m^{-1}(U_{-I'})\cap  \sigma_{I'}(L'_{\e})\cap Z \big]
\cup \widetilde{\rho_{I'}}\cup \sigma_{I'}(\phi)\bigg)
\subset
\big(Z\cap \C_{z',\e}^{m-1}\big).
$$
Then, by Lemma \ref{local lemma},
the image of the homomorphism
$$
\pi_1(\big[\pi_m^{-1}(U_{-I'})\cap  \sigma_{I'}(L'_{\e})\cap Z \big]
\cup \widetilde{\rho_{I'}}\cup \sigma_{I'}(\phi),\dot{z})
\to
\pi_1(Z\cap \C_{z',\e}^{m-1},\dot{z})
$$
is generated by
$\Pi_{I'}$.
Thus we have
$$
\pi_1(Z\cap \C_{z',\e}^{m-1},\dot{z})=\langle
\pi_1(Z_{\e}',\dot{z}),\Pi_{I'}
\rangle_{I'\in (\Z/p\Z )^{m-1}}
$$
by van Kampen's theorem. 
By Theorem \ref{th:reduction},
$\iota_{\infty}(\Ad_{I'}(\widetilde{\rho_0'}))$ belongs to $\Pi_{I'}$ and
$$
\pi_1(Z\cap \C_{z',\e}^{m-1},\dot{z})=\langle
\widehat{\rho_1}, 
\dots, \widehat{\rho_{m-1}},
\Ad_I(\widetilde{\rho_0})
\rangle_{I\in (\Z/p\Z )^{m}}.
$$
Hence we have
$$
\pi_1(Z,\dot{z})=
\langle
\widehat{\rho_1}, 
\dots, \widehat{\rho_{m-1}},\widehat{\rho_{m}},
\Ad_I(\widetilde{\rho_0})
\rangle_{I\in (\Z/p\Z )^{m}},
$$
which completes Proof of Theorem \ref{generator for Z}.
\end{proof}

\begin{proof}[Proof of Theorem \ref{generator of fund of X}]
Let $\gamma$ be a loop representing an element in the fundamental group $\pi_1(X,\dot x)$ and 
$\widetilde{\gamma}$ be 
its lifting starting from $\dot{z}$ for the morphism $\f$.
The end point of $\widetilde{\gamma}$ is not 
$\dot{z}$ in general, but there exists an element $J=(j_1,\dots,j_m) 
\in (\Z/p\Z )^m$ such that it is equal to $\sigma_J(\dot{z})$.
Thus the composite 
$\widetilde{\gamma}\cdot \widetilde{\rho_J}^{-1}$
is a loop in $Z$ and it defines an element of
$\pi_1(Z,\dot z)$.
By Theorem \ref{generator for Z}, we have
$$
\widetilde{\gamma}\cdot \widetilde{\rho_J}^{-1} \in
\langle
\widehat{\rho_1}, 
\dots, \widehat{\rho_{m}},
\Ad_I(\widetilde{\rho_0})
\rangle_{I\in (\Z/p\Z )^{m}},
$$
and 
$$
 \f_*(\widetilde{\gamma}\cdot \widetilde{\rho_J}^{-1})
 \in
 \langle
 \rho_1^p, 
 \dots, \rho_{m}^p,
 (\rho_1^{i_1}\cdots \rho_m^{i_m})\rho_0 
 (\rho_1^{i_1}\cdots \rho_m^{i_m})^{-1}
 \rangle_{I\in (\Z/p\Z )^{m}}
 $$
for the representatives $i_k$ $(1\le k\le m)$  
chosen from $\{0,1,\dots,p-1\}$.
Therefore,  we have
 $$
 \f_*(\widetilde{\gamma}\cdot \widetilde{\rho_J}^{-1})
 =\gamma (\rho_1^{j_1}\cdots \rho_m^{j_m})^{-1},
 $$
if we choose the representatives $j_k$ $(1\le k\le m)$ 
from $\{0,1,\dots,p-1\}$, 
and the element $\gamma$ in $\pi_1(X,\dot x)$ 
belongs to 
the group $\langle\rho_0, \rho_1, \dots, \rho_{m}\rangle$.
\end{proof}

 \ifdefined\ifmaster
 \else
\end{document}
\fi

 \ifdefined\ifmaster
 \else
 \documentclass[12pt]{amsart}
 \input preamble.tex
\usepackage{amsmath}	
\begin{document}

\fi

\section{Monodromy representation} 
In this section, we study the monodromy representation of $\cF_C^{p,m}(a,B)$, 
which is the homomorphism from the fundamental group $\pi_1(X,\dot x)$ 
to the general linear group of the space of solutions to 
$\cF_C^{p,m}(a,B)$ around $\dot x$ given by the analytic continuation.

In Theorem \ref{generator of fund of X}, we have proved that
the fundamental group $\pi_1(X,\dot x)$ is generated by the loops
$\rho_0$, $\rho_1$,$\dots$, $\rho_m$.
In the following, we 
give an explicit formula for the circuit matrix $M_k$ 
(see \eqref{eq:circ-matrix}) along  
$\rho_k$ for each $k=0,1,\dots,m$.

\subsection{Fundamental system of solutions to  $\cF_C^{p,m}(a,B)$}
Let $a=\tr(a_1,\dots,a_p)$ be an element of $\C^p$ and
$B$
be a $p\times m$ matrix.
From now on, we assume the following conditions:
\begin{equation}
\label{eq:non-integral}
a_i-\sum_{k=1}^m b_{j_k,k},\notin \Z, 
\end{equation}
\begin{equation}
\label{eq:non-integral-addition}
b_{j,k}-b_{j',k}\notin \Z,
\end{equation}
where $1\le i\le p$, $1\le j<j'\le p$, $1\le k\le m$, and  
$J=(j_1,\dots,j_m)\in [p]^m$.
To give explicit expressions of independent
solutions to the system $\cF_C^{p,m}(a,B)$,  
we introduce maps $\eta_{j_k}$, matrices $(B)_J$ and vectors $a_J$ as follows.

\begin{definition}
\label{def:eta-aJBJ}
\begin{enumerate}
 \item 
For an integer $j\in [p]$ and a vector 
${\bf b}=\tr(b_1,\dots, b_p)\in \tr\C^p$, 
we set 
\begin{equation}
\label{index reflection column}
\eta_{j}({\bf b})={\bf b}+(1-b_j)\one_p.
\end{equation}
\item
Let $B=({\bf b}_1,\dots, {\bf b}_m)$ be a $p\times m$ matrix.
For $J=(j_1,\dots, j_m)\in [p]^m$, 
we set
$$
\mu_{J,k}=1-b_{j_k,k},\quad 
\mu_J=\sum_{k=1}^m\mu_{J,k}=\sum_{k=1}^m(1-b_{j_k,k}),
$$
and
\begin{equation}
\label{eq:fund-solutions1}
a_J=a+\mu_J\one_p, \quad
(B)_J=\Big(\eta_{j_1}({\bf b}_1),\dots,
\eta_{j_m}({\bf b}_m)\Big),
\end{equation}
where $\one_p=\tr(1,\dots,1)$.
\item
We consider the following condition:
\begin{equation}
\label{relax the condition for B}
\text{ there exists an integer $j \in [p]$ such that $b_{j,k}=1$ for each $k$.} 
\end{equation}
For a $p\times m$ matrix $B$ satisfying the condition
\eqref{relax the condition for B},
we extend the definition of
${F}_C^{p,m}(a,B;x)$ by the series \eqref{eq:HGS ser}.


\end{enumerate}

\end{definition}



\begin{remark}
\begin{enumerate}
\item 
It is easy to see that if $B$ satisfies the condition
\eqref{relax the condition for B} in Definition \ref{def:eta-aJBJ}.(3), 
then $(B)_J$ also satisfies this condition. 
Note that if a parameter $B$ is standard,
it satisfies this condition.

\item 
The definition of $B_J$ in the paper \cite{KMOT} is different from
$(B)_J$.
They are equal up to permutation of row in each column. 
\end{enumerate}
\end{remark}
\begin{proposition}[{\cite[Proposition 4.4]{KMOT}}]
\label{fact:fund-solutions}
For $k=1,\dots, m$, we have
$$
\Big(\prod_{k=1}^m x_k^{\mu_{J,k}}\Big)^{-1}
\ell_k(a,B)\Big(\prod_{k=1}^m x_k^{\mu_{J,k}}\Big)
=\ell_k(a_J,B_J).
$$
Therefore, under the conditions \eqref{eq:non-integral} and 
\eqref{eq:non-integral-addition}, 
there are linearly independent $p^m$ solutions 
\begin{equation}
\label{eq:fund-solutions}
\Phi_J(a,B;x)
=\Big(\prod_{k=1}^m x_k^{\mu_{J,k}}\Big)
{F}_C^{p,m}\Big(a_J,(B)_J;x\Big)
\end{equation}
to $\cF_C^{p,m}(a,B)$ 
on a small neighborhood in $\D$ 
indexed by $J\in [p]^m$.
If $J=(p,\dots, p)$, then we have
\begin{equation}
\Phi_{(p,\dots,p)}(a,B;x)
=
F_C^{p,m}\Big(a,B;x\Big).
\end{equation}
\end{proposition}

\subsection{Integral expressions and gamma factors}
\label{determination of gamma factors}
In this subsection, we use some consequences of twisted homology theory.
As for the case $m=1$, see \cite{Ma2}.

We define the monodromy field $K$ by the rational function field
\begin{equation}
\label{def of monodromy field}
K=\Q(\a_i,\b_{j,k})_{\substack{1\leq i,j\leq p, 
1\leq k \leq m}},
\end{equation}
where $\a_i=\ex(a_i) $ and $\b_{j,k}={\bf e}(b_{j,k})$.
In \cite{KMOT3},
we introduce a twisted cycle $\Delta_J$ with coefficients in $K$
for $J\in [p]^m$. 
For a standard parameter $B$, we set
\begin{equation}
\label{gamma normalized HGF}
\Phi^{\G}_J(a,B;x)=
\int_{\Delta_J}\!\!
\xi
\cdot \dfrac{dt}{t}.
\end{equation}
where $\xi$ is defined in Proposition \ref{fact:int-rep}.
The twisted homology group associated with $F_C^{p,m}(a,B;x)$ is spanned by
$\Delta_J$ ($J\in [p]^m$).
Then we will prove the following proposition in \cite{KMOT3}.
\begin{proposition}
\label{int expression}
Let $J=(j_1,\dots, j_m)$ and we set $e_j=b_{j,1}+\cdots+b_{j,m}-a_j-m$.
Then we have the equality:
\begin{equation}
\label{relation between Phi and Phi Gamma}
\Phi^{\G}_J(a,B;x)=c^{\G}_J\Phi_J(a,B;x),
\end{equation}
where ${\Phi}_J(a,B;x)$ is defined in 
\eqref{eq:fund-solutions},
$b^*_{j,k}$ is the $(j,k)$ entry of $(B)_J=(b^*_{j,k})_{j,k}$
defined in  \eqref{eq:fund-solutions1}, 
and
\begin{align}
\label{gamma factor for J}
c^{\G}_J=&c^{\G}_J(a,B)
=
\prod_{j=1}^p
\dfrac
{\G(1+e_j)\prod_{\substack{1\leq k \leq m\\ j_k\neq j}}\G(1-b^*_{j,k})}
{\G(1+e_j+\sum_{\substack{1\leq k \leq m\\ j_k\neq j}}(1-b^*_{j,k}))}.
\end{align} 
\end{proposition}

Let $\mathcal V$ be the 
$K$-linear span of
$\{\Phi^{\G}_J(a,B;x)\}_{J\in [p]^m}$.
Then $\mathcal V$ is identified with the 
twisted homology group via the correspondence 
$\Phi^{\G}_J(a,B;x)$ $\leftrightarrow$ $\Delta_J$.
By arraying these functions, we set a locally holomorphic 
column vector valued function
$$
\Psi^{\G}=\tr(
\Phi^{\G}_J(a,B;x))_{J\in [p]^m}.
$$
We identify $K^{[p]^m}$ with
$\mathcal V$ by the map
\begin{equation}
\label{hypergeom identification}
K^{[p]^m}\ni v \mapsto
v\Psi^{\G} \in \mathcal V.
\end{equation}

Let $\rho$ be a loop with base point $\dot x$ in $X$. 
For an element $\psi$ of $\mathcal{V}$ and  
a vector $\Psi$ with entries in $\mathcal{V}$, 
$\psi_\rho$ and $\Psi_\rho$ denote 
the analytic continuations of $\psi$ and $\Psi$ along $\rho$, 
which depend only on the homotopy class of $\rho$. 
By theory of twisted homology, 
the space $\mathcal V$ is stable under analytic continuations 
along loops with base point $\dot x$ in $X$.

\begin{definition}
We have a linear map 
$$\mathcal{M}_\rho:\mathcal{V}\ni \psi \mapsto \psi_\rho\in \mathcal{V},$$
which is called the circuit transformation of $\mathcal{V}$ along 
a loop $\rho$ with base point $\dot x$ in $X$.
By using the identification \eqref{hypergeom identification}, 
we represent the circuit transformation $\mathcal{M}_\rho$ by 
a matrix $M_\rho$, i.e.,  there uniquely exists a  matrix 
$M_\rho=M_\rho(a,B)$
in $GL(K^{[p]^m})$ such that 
$$
\Psi_{\rho}^{\G}=M_{\rho}\Psi^{\G}.
$$
The matrix $M_{\rho}$ is called the circuit matrix along $\rho$ 
with respect to $\Psi^{\G}$. 
The monodromy representation of $\mathcal{F}_C^{p,m}(a,B)$ is 
defined as the homomorphism 
$$\pi_1(X,\dot x)\ni \rho\mapsto \mathcal{M}_\rho\in GL(\mathcal{V})
\quad (\textrm{or }\ \pi_1(X,\dot x)\ni \rho\mapsto M_\rho\in GL(K^{[p]^m})),
$$
where $GL(\mathcal{V})$ denotes the general linear group of $\mathcal{V}$ 
over $K$.
\end{definition}

Since the fundamental group  $\pi_1(X,\dot x)$ is generated by the loops
$\rho_0$, $\rho_1$,$\dots$, $\rho_m$ defined in 
\eqref{eq:loops} and \eqref{eq:loop0},
the 
monodromy representation of $\mathcal{F}_C^{p,m}(a,B)$ 
is determined by the circuit matrices 
\begin{equation}
\label{eq:circ-matrix}
M_k=M_{\rho_k} \quad (k=0,1,\dots,m).
\end{equation}
By the following proposition,
the mutually commutative right actions of $M_1, \dots, M_m$ are simultaneously
diagonalized.
\begin{proposition}
\label{orthogonality and diagonal 1}
We fix an element $J=(j_1,\dots,j_m)$ in $[p]^m$. The vector 
$\bold e_{J}=(0,\dots, \overset{J \text{-th }}1,\dots, 0)\in K^{[p]^m}$ 
is an eigenvector
for each circuit matrix $M_k$ $(k=1,\dots,m)$ of 
eigenvalue $\b_{j_k,k}^{-1}$.
In other words, we have ${\bf e}_JM_k=\b_{j_k,k}^{-1}{\bf e}_J$
for $k=1,\dots, m$. Note that $\b_{j_k,k}\ne \b_{j'_k,k}$ for 
$1\le j_k<j'_k\le p$, $1\le k\le m$.
\end{proposition}
\begin{proof}
Since 
${\bf e}_J=\Phi^\G_J(a,B;x)\in \mathcal V$
is the product of 
$\prod_{k=1}^m x^{1-b_{j_k,k}}$ 
and a holomorphic function of $x$ with a non-zero constant term.
The function $\Phi^{\G}_J(a,B;x)$ corresponds to a 
$\beta_{j_k,k}^{-1}$-eigenvector of
the circuit matrix $M_k$ $(k=1,\dots,m)$.
\end{proof}
\begin{definition}
\label{eigen components}
\begin{enumerate}
\item
For $J=(j_1,\dots, j_m)\in [p]^m$, $\mathcal V_J$ denotes
the simultaneous eigenspace in $\mathcal V$ 
of the eigenvalues
$\beta_{j_1,1}^{-1},\dots, \beta_{j_{m},m}^{-1}$ 
for the right actions of $M_1, \dots, M_{m}$.
Then we have
$$
\mathcal V
=\underset{J\in [p]^{m}}
\sum
\mathcal V_{J},\quad
\mathcal V_J
=K\bold e_J \subset K^{[p]^m}.
$$
 \item 
For $j\in [p]$ and $1\leq k\leq m$, the 
$\beta_{j,k}^{-1}$ eigenspace of $M_k$ in $\mathcal V$ 
is $p^{m-1}$ dimensional and is denoted by $\mathcal V_{j,k}$. 
Then we have
\begin{equation}
\label{direct sum Mk}
\mathcal V=\underset{j \in [p]}\sum \mathcal V_{j,k}.
\end{equation}
We set $\mathcal V_{j,m}=\mathcal V_j$.
For $j\in [p]$, we define an isomorphism $E_j:K^{[p]^{m-1}}\to \mathcal V_j$
by
\begin{equation}
\label{nat incl add j} 
E_j:K^{[p]^{m-1}}\ni {\bf e}_{J'} \mapsto 
{\bf e}_{J'(j)}\in \mathcal V_j \subset \mathcal V = K^{[p]^m},
\end{equation}
where
$J'(j)=(j_1,\dots, j_{m-1},j)$. 
Via the identification $\mathcal V \simeq K^{[p]^m}$ 
in \eqref{hypergeom identification},
we define $P_j:\mathcal V \to K^{[p]^{m-1}}$
as the projection 
$$
K^{[p]^m}\ni (v_J)_J\mapsto (v_{J'(j)})_{J'\in [p]^{m-1}}\in K^{[p]^{m-1}}.
$$

\end{enumerate}
\end{definition}


\subsection{Monodromy representation and the infinitesimal inclusion}
Recall that $X'$ be the space defined in 
\eqref{one dim lower}.
Let $\rho_0',\rho_1',\dots, \rho_{m-1}'$ be the generators of
$\pi_1(X',\dot{x}')$ defined as the case of $m-1$.
By using the infinitesimal inclusion
$$
\iota_{\infty}:\pi_1(Z',\dot{z}') \to \pi_1(Z,\dot{z})
$$
in \eqref{inif incl X}, 
we define $N_0=M_{\iota_{\infty}(\rho_0')}$ as
the circuit matrix along $\iota_{\infty}(\rho_0')$.
Then we have the relation
\begin{equation}
\label{eq:reduction}
N_0= M_0( M_m M_0 M_m^{-1})\cdots ( M_m^{p-1} M_0 M_m^{1-p})
\end{equation}
of circuit matrices by Theorem \ref{th:reduction} (1).
\begin{proposition}
\label{stable under N_0}
Under the identification \eqref{hypergeom identification}, 
we have
\begin{align}
&\mathcal V_{j}M_k\subset \mathcal V_{j},\quad
\mathcal V_{j}N_0\subset \mathcal V_{j}\quad (1\leq j\leq p,\  k=1,\dots, m).
\end{align}
As a consequence, the right actions $M_1, \dots, M_{m}$ and $N_0$ on 
the space $\mathcal V$
induce those on $\mathcal V_{j}$.
\end{proposition}
\begin{proof}
 
By Theorem \ref{th:reduction} (2) and Corollary \ref{circular relation}, 
the matrices $M_k$ ($k=1,\dots, m$) and $N_0$ commute with 
$M_m$. Therefore for $v'\in K^{[p]^{m-1}}$, we have
\begin{align*}
&E_j(v')M_kM_m=E_j(v')M_mM_k=\beta_{j,k}^{-1}E_j(v')M_k,
\\
&E_j(v')N_0M_m=E_j(v')M_mN_0=\beta_{j,k}^{-1}E_j(v')N_0,
\end{align*}
and $E_j(v')M_k, E_j(v')N_0 \in \mathcal V_{j}$.
\end{proof}
By Proposition \ref{stable under N_0},
the space $\mathcal V_j$
is stable under the right action of 
$N_0$, which can be expressed as that of
$N_{0,j} \in GL(K^{[p]^{m-1}})$,
i.e.
\begin{align}
\label{E_j and v'}
E_j(v'N_{0,j})=E_j(v')N_{0}
\end{align}
for $v'\in K^{[p]^{m-1}}$.


We consider the restrictions of
hypergeometric functions  
to $V(x_m)$ defined in \eqref{def of vxm}.
For a matrix $B=({\bf b}_1,\dots, {\bf b}_m)$, 
we set a matrix $B'=({\bf b}_1,\dots, {\bf b}_{m-1})$.
By definition \eqref{eq:HGS ser} of the series $F_C(a,B;x)$,
we have
\begin{equation}
\label{eqn:series restriction original} 
F_C^{p,m}(a,B;x)\mid_{x_m=0}=F_C^{p,m-1}(a,B';x'),
\end{equation}
where $*\mid_{x_m=0}$ denotes the restriction of a function 
to the set $V(x_m)$ and
$x'=(x_1, \dots, x_{m-1})$.
We generalize \eqref{eqn:series restriction original} 
into
\begin{equation}
\label{system of restrictions} 
\big(x^{b_{j,m}-1}\Phi_{J'(j)}(a,B;x)\big)\mid_{x_m=0}
=\Phi_{J'}(a(j),B',x'),
\end{equation}
where
$$
a(j)=a_{{\bf p}'(j)}=a+(1-b_{j,m}){\bf 1}_p, \quad {\bf p}'(j)=
(\underset{\text{$(m-1)$-times}}{\underbrace{p,\dots,p} },j).
$$

We compare the monodromy representation on 
$\mathcal V_j$ in Proposition \ref{stable under N_0}
and that of $\cF_C^{p,m-1}(a(j),B')$ via the infinitesimal inclusion 
$\iota_{\infty}:\pi_1(Z',\dot{z}') \to \pi_1(Z,\dot{z})$ using twisted homology
theory. To use twisted homology theory for a system of differential
equations $\cF_C^{p,m-1}(a(j),B')$, we define
$\Phi^{\Gamma}_{J'}(a(j),B',x')$ by
$$
\Phi^{\Gamma}_{J'}(a(j),B',x')={c'}^{\G}_{J'}\Phi_{J'}(a(j),B',x'),
$$
where $(B')_{J'}=({b'}^*_{j,k})_{j,k}$, 
$e_\ell'=e_\ell+b_{j,m}-b_{\ell,m}$, and
\begin{align}
\label{gamma factor for J}
{c'}^{\G}_{J'}=
\prod_{j=1}^p
\dfrac
{\G(1+e_j')\prod_{\substack{1\leq k \leq m-1\\ j_k\neq j}}
\G(1-{b'}^*_{j,k})}
{\G(1+e_j'+\sum_{\substack{1\leq k \leq m-1\\ j_k\neq j}}
(1-{b'}^*_{j,k}))}.
\end{align} 
Let $K'$ be the field defined by
$$
K'=\Q(\a_i,\b_{j,k})_{\substack{1\leq i,j\leq p, 
1\leq k \leq m-1}},
$$
and $\mathcal V'_j$ be the $K'$-vector space generated by
$\Phi_{J'}^{\G}(a(j),B',x')$ for $J'\in [p]^{m-1}$.
Using the basis
$\{\Phi_{J'}^{\G}(a(j),B',x')\}$,
we identify the space $\mathcal V'_j$ with
${K'}^{[p]^{m-1}}$ similarly to the identification 
\eqref{hypergeom identification}.
The circuit matrix along $\rho'_0$ with respect to this basis
is denoted by $M'_{0,j}$.

We define an isomorphism 
\begin{equation}
\label{isomorphism with monodromy action} 
\varrho_*:\mathcal V'_{j}\otimes_{K'}K 
\ni
\Phi_{J'}^{\G}(a(j),B',x')
\mapsto
\Phi_{J'(j)}^{\G}(a,B;x)\in
\mathcal V_{j}
\end{equation} 
for $J'=(j_1, \dots, j_{m-1})\in [p]^{m-1}$. 
\begin{proposition}
\label{comparizon of gamma factor}
\begin{enumerate}
 \item 
The quotient
$$
\gamma_{j}=\dfrac
{c^{\G}_{J'(j)}}
{{c'}^{\G}_{J'}}
=\prod_{1\leq \ell \leq p, \ell\neq j}
\dfrac
{\G(1+e_\ell)\G(b_{j,m}-b_{\ell,m})}
{\G(1+e'_\ell)}
$$
depends only on $j$, and is independent of $J'$.

\item
Therefore 
under the isomorphism 
$\varrho_*$,
we have
$$
{M'_{0,j}}=
{N_{0,j}}.
$$
The kernel of $M'_{0,j}$ 
and that of 
$N_{0,j}$ 
coincide via the isomorphism $\varrho_*$.

\end{enumerate}
\end{proposition}
\begin{proof}
(1) In the expression \eqref{gamma factor for J}
of ${c'}^{\G}_{J'}$, we check the dependence on $J'$ of 
the factors 
$$
\G(1+e_\ell+\sum_{\substack{1\leq k \leq m\\ j_k\neq \ell}}(1-{b'}^*_{\ell,k}))
$$
for 
$c^{\G}_{J'(j)}$ and for ${c'}^{\G}_{J'}$.
Since
\begin{align*}
&1+e'_\ell+\sum_{\substack{1\leq k \leq m-1\\ j_k\neq \ell}}(1-{b'}^*_{\ell,k})
=1+e_\ell+\sum_{\substack{1\leq k \leq m\\ j_k\neq \ell}}(1-b^*_{\ell,k}),
\end{align*}
these factors cancel, thus we have the statement (1).

\smallskip\noindent
(2) 
We set
\begin{align*} 
&\Psi_{j}=\tr(
\Phi_{J'(j)}(a,B;x))_{J'\in [p]^{m-1}},\quad
\Psi_{j}^{\G}=\tr(
\Phi^{\G}_{J'(j)}(a,B;x))_{J'\in [p]^{m-1}}.
\\
&\Psi'_{j}=\tr(
\Phi_{J'}(a(j),B';x'))_{J'\in [p]^{m-1}},\quad
{\Psi'}^{\G}_{j}=\tr(
\Phi^{\G}_{J'}(a(j),B';x'))_{J'\in [p]^{m-1}}.
\end{align*}
Then by (1), we have
$$
\Psi^{\G}_{j}=C\Psi_{j},\quad
{\Psi'}_{j}^{\G}=C'\Psi'_{j},\quad C=\gamma_jC'
$$
with $C=\diag(c^\G_J), C'=\diag({c'}^\G_{J'})$.
 The coordinate $x_m=\e $ is constant under
the image of the infinitesimal inclusion
and the circuit matrix is kept invariant under the small deformation of 
the base point.
Therefore by \eqref{system of restrictions},
there exists a matrix  $M^*_{0,j}$ such that
\begin{align*}
&\Psi_{j,\iota_{\infty}(\rho'_0)}=M^*_{0,j}\Psi_{j},\quad
\Psi'_{j,\rho_0'}=M^*_{0,j}\Psi'_{j}.
\end{align*}
On the other hand, by the definitions of
$N_{0,j}$ and $M'_{0,j}$, we have
\begin{align*}
&\Psi^{\G}_{j,\iota_{\infty}(\rho'_0)}=N_{0,j}\Psi^{\G}_{j},\quad
{\Psi'}^{\G}_{j,\rho_0'}=M'_{0,j}{\Psi'}^{\G}_{j}.
\end{align*}
These relations imply $N_{0,j}=M'_{0,j}$.
\end{proof}

\subsection{Monodromy invariance of the intersection form}
\label{subsec:intersection}
We recall some properties
on intersection theory for twisted 
homology groups. 
Let $k\mapsto  k^{\vee}$ be an involution 
of $K$ obtained by replacing 
$(a,B)$ by $(-a,-B)$.
It induces an involution $M\mapsto M^\vee$ on $GL_n(K)$.
For $K$-vector spaces $L$ and $M$, 
 an additive map $\varphi:L \to M$ is called
a $K$-conjugate map, if $\varphi(kv)=k^{\vee}\varphi(v)$ for any $k\in K$ and
$v\in L$.
To study the intersection pairing of twisted homology groups, 
we consider the fundamental solutions 
$\{\Phi^{\G}_J(-a,-B;x)\}_{J\in [p]^m}$ 
to $\cF_C^{p,m}(-a,-B)$ and define 
$\Psi^{\vee}$ by
\begin{align*}
&\Psi^{\vee}=\Psi(-a,-B)=\ ^t(
\Phi^{\G}_J(-a,-B;x))_{J\in [p]^m}.
\end{align*}
Let $\mathcal V^{\vee}$
be the $K$-vector space spanned by 
$\{\Phi^{\G}_J(-a,-B;x)\}_{J\in [p]^m}$. It is identified with $K^{[p]^m}$.
By the intersection form of twisted homology groups 
(\cite[Lemma 5.2]{Ma1}), we have the $K$-bilinear form
on $\mathcal V\times \mathcal V^{\vee}$.
Using this $K$-bilinear form and the natural $K$-conjugate map 
$$
\mathcal V \ni \varphi(a,B;x)\mapsto \varphi(-a,-B;x)\in \mathcal V^{\vee},
$$
we obtain the non-degenerate skew bilinear paring $\langle *,*\rangle$ 
on $\mathcal V \times \mathcal V$. 
Via the identification \eqref{hypergeom identification},
the corresponding skew symmetric pairing on $K^{[p]^m}\times K^{[p]^m}$
is also denoted by $\langle *,*\rangle$.
Then it satisfies
\begin{equation}
\label{skew bilinear}
\langle ku , w \rangle=k\langle u , w \rangle=
\langle u , k^\vee w \rangle, 
\ (u,w \in K^{[p]^m}, k\in K).
\end{equation}
Moreover it is invariant under the right multiplication of
any circuit matrix, i.e.
\begin{equation}
\label{intersection invariant}
\langle uM_{\rho} , wM_{\rho} \rangle=\langle u , w \rangle. 
\end{equation}
By the identification 
\eqref{hypergeom identification},
the intersection form is given as
$$
\langle u , w \rangle=u \cH\ ^t w^\vee
$$
for $u,w \in K^{[p]^m}$.
Since ${\bf e}_J^{\vee}={\bf e}_J$, 
the intersection matrix $\cH$ is given by
\begin{equation}
\label{intersection matrix total}
\cH
=(\langle{\bf e}_I,{\bf e}_J\rangle)_{I,J\in [p]^m}.
\end{equation}
Then by properties \eqref{skew bilinear} 
and \eqref{intersection invariant}, we have
\begin{equation}
\label{eq:keep-IM1}
\ ^t\mathcal H^{\vee}=\mathcal H,\quad M_\rho\cH\tr M_\rho^{\vee}=\cH.
\end{equation} 
\begin{proposition}
\label{orthogonality and diagonal}
If $I$ and $J$ are different elements in $[p]^m$,  
then  ${\bf e}_I\mathcal H\ ^t{\bf e}^{\vee}_J=0$. 
In other words, we have the orthogonal property
between the bases $\{\Phi^{\G}_J(a,B;x)\}_{J\in [p]^m}$ and
$\{\Phi^{\G}_J(-a,-B;x)\}_{J\in [p]^m}$, i.e.
$\cH$ is a diagonal matrix.
\end{proposition}
\begin{proof}
Let $I=(i_1,\dots,i_m)$ and $J=(j_1,\dots,j_m)$ be different 
elements in $[p]^m$.
Choose an index $k$ such that $i_k \neq j_k$. Then we have
$$
\tr M_k^{\vee}\tr{\bf e}_J=\tr ({\bf e}_JM_k)^{\vee}=
(\b_{j_k,k}^{-1}\tr{\bf e}_J)^{\vee}=
\b_{j_k,k}\tr{\bf e}_J,
$$
and
$$
{\bf e}_I\mathcal H\tr{\bf e}_J=
{\bf e}_IM_k \mathcal H\tr M_k^{\vee}{\bf e}_J=
{\bf e}_I\b_{i_k,k}^{-1} \mathcal H \b_{j_k,k}{\bf e}_J.
$$
Since $\b_{i_k,k}^{-1} \b_{j_k,k} \neq 1$, we have 
${\bf e}_I\mathcal H\tr{\bf e}_J=0$.
\end{proof}
We define the intersection matrix 
$\mathcal H_j$ on $K^{[p]^{m-1}}$ by
\begin{equation}
\label{def of H_j} 
\mathcal H_j=\big(E_j({\bf e}_{I'})\,\mathcal H \tr E_j({\bf e}_{J'}\big)^{\vee})
_{I',J'\in [p]^{m-1}},
\end{equation}
where $E_j$ is given in \eqref{nat incl add j}.
Note that $\mathcal H_j$ is a submatrix of $\mathcal H$.
%

\subsection{$H$-reflections}
\label{subsec:H reflection}
\begin{definition}[$H$-reflection]
\label{def:H reflection}
Let $n$ be a positive integer and
$H$ be an element of $GL_n(K)$ satisfying 
$\ ^tH^{\vee}=H$.
A matrix $M\in GL_n(K)$ is called
an $H$-reflection, if it
satisfies 
the following conditions. 
\begin{enumerate}
 \item 
The image 
$$
\Im(r(M-I_n))=\{{u}(M-I_n)\mid {u}\in K^n\}
$$
of $r(M-I_n)$ 
is a one-dimensional $K$-vector space.
Here $I_n$ is the identity matrix in $GL_n(K)$.
\item
The following equality holds:
\begin{equation}
\label{eq:keep-IM}
M H\tr M^\vee=H.
\end{equation} 
\end{enumerate}
Let $M$ be an $H$-reflection.
A non-zero element of $\Im(r(M-I_n))$ is called a reflection
vector of $M$. The determinant $\det(M)$ of $M$ is 
called the reflection eigenvalue of $M$, and is denoted by 
$\lambda=\lambda(M)$.
\end{definition}
The kernel $\Ker(r(M-I_n))=\{{w}\mid 
{w}M={w}\}$ of the right action of $M-I_n$
is an $(n-1)$-dimensional $K$-vector space.

\begin{proposition}
\label{prop: H-reflection}
Let
$M$ be an $H$-reflection
in Definition \ref{def:H reflection} and 
$\lambda=\lambda(M)$ be its reflection eigenvalue.
We choose a reflection vector ${\bf v}$ of $M$.
\begin{enumerate}
\item
The one-dimensional space $\Im(r(M-I_n))$ is stable under the 
action of $M$, and
we have ${\bf v}M=\lambda {\bf v}$.
 \item 
We have
$$
\Ker(r(M-I_n))=
\{{u}\mid {u}H\tr {\bf v}^{\vee}=0\}.
$$
Conversely, if ${\bf v'} \in K^n$ satisfies the condition
$$
\Ker(r(M-I_n))=
\{{u}\mid {u}H\tr {\bf v'}^{\vee}=0\},
$$
then ${\bf v'}$ is a reflection vector.

\item
We choose an element ${\bf v^*}\in K^n$ such that
${\bf v^*}(M-I_n)={\bf v}$.
Then we have
${\bf v^*}H\tr {\bf v}^{\vee}\neq 0$
and
\begin{equation}
\label{H reflection}
M=I_n+
\dfrac{H\tr {\bf v}^{\vee}}{{\bf v^*}H\tr {\bf v}^{\vee}}{\bf v}.
\end{equation}

\item
If ${\bf v}H\tr {\bf v}^{\vee}\neq 0$, then $\lambda\neq 1$ and
we may choose ${\bf v^*}$ as $\frac{1}{\lambda-1}{\bf v}$,
and 
\begin{equation}
\label{H reflection 2}
M=I_n+
(\lambda-1)\dfrac{H\tr {\bf v}^{\vee}}{{\bf v}H\tr {\bf v}^{\vee}}{\bf v}.
\end{equation}
\end{enumerate}
\end{proposition}

\begin{proof}
(1) 
The first assertion follows from the equality
$M(M-I_n)=(M-I_n)M$.
Then the right action of $M$ acts on the exact sequence
$$
0\to \Ker(r(M-I_n)) \to K^n \to \Im(r(M-I_n)) \to 0.
$$
Since the action of $M$ on $\Ker(r(M-I_n))$ is identity
and the action of $M$ on $\Im(r(M-I_n))$ is the multiplication by 
the determinant of $M$, we have ${\mathbf v} M=\lambda{\mathbf v}$.

\smallskip\noindent
(2) 
Since $H$ and $M$ are invertible, ${u}$ is contained in 
$\Ker(r(M-I_n))$ if and only if
$$
{u}(M-I_n)H \tr M^{\vee}\tr {x}^{\vee}=0
$$
for any $x\in K^n$. By the equality \eqref{eq:keep-IM}, we have
$$
{u}(M-I_n)H\tr M^{\vee}\tr {x}^{\vee}
={u}H(I_n-\tr M^{\vee})\tr {x}^{\vee}
={u}H\tr\bigg({x}(I_n-M) \bigg)^{\vee}.
$$
Since ${\bf v}$ is a reflection vector of $M$,
${u}$ is contained in $\Ker(r(M-I_n))$ if and only if
${u}H\tr {\bf v}^{\vee}=0$. Since ${\bf v}$ is 
a generator of the one-dimensional space
$$
\{w\in K^n\mid {u}H\tr w^{\vee}=0 \text{ for any }{u} 
\in \Ker(r(M-I_n))\},
$$
we have the second statement.

\smallskip\noindent 
(3)
Since ${\bf v^*}$ is not contained in $\Ker(r(M-I_n))$,
we have ${\bf v^*}H\tr {\bf v}^{\vee}\neq 0$.
We check the equality \eqref{H reflection}
by applying the right hand side of \eqref{H reflection 2}
to an element ${u}$ in $K^n$ from the right.
If ${u}$ is contained in $\Ker(r(M-I_n))$, then we have
$$
{u}\bigg(I_n+
\dfrac{H\tr {\bf v}^{\vee}}{{\bf v^*}H\tr {\bf v}^{\vee}}{\bf v}\bigg)
=
\bigg({u}+
\dfrac{{u}H\tr {\bf v}^{\vee}}{{\bf v^*}H\tr {\bf v}^{\vee}}{\bf v}\bigg)
={u}={u}M.
$$
As for ${\bf v^*}$, we have
\begin{align*}
{\bf v^*}\bigg(I_n+
\dfrac{H\tr {\bf v}^{\vee}}{{\bf v^*}H\tr {\bf v}^{\vee}}{\bf v}\bigg)
=&
\bigg({\bf v^*}+
\dfrac{{\bf v^*}H\tr {\bf v}^{\vee}}{{\bf v^*}H\tr {\bf v}^{\vee}}{\bf v}\bigg)
={\bf v}+{\bf v^*}
\\
=&{\bf v^*}(M-I_n)+{\bf v^*}={\bf v^*}M.
\end{align*}
Since $\Ker(r(M-I_n))$ and ${\bf v^*}$ generate $K^n$, we have the identity
\eqref{H reflection}.

\smallskip\noindent
(4) 
If ${\bf v}H\tr {\bf v}^{\vee}\neq 0$, then ${\bf v} \notin \Ker(r(M-I))$.
Since ${\bf v}$ is an eigenvector of $M$ which is not contained in
$\Ker(r(M-I))$, the reflection eigenvalue is not $1$.
If the reflection eigenvalue $\lambda=\lambda(M)$ is not $1$,
$
{\bf v}(M-I_n)=(\lambda-1){\bf v}
$
and we may choose ${\bf v^*}$ as $\frac{1}{\lambda-1}{\bf v}$.
\end{proof}

The following lemma is a key to obtain the circuit matrix $M_0$ 
along $\rho_0$\begin{lemma}.
\label{lem:1-eigen-space}
The circuit matrix $M_0\in GL_{p^m}(K)$ is an $\mathcal H$-reflection.
\end{lemma}
\begin{proof}
We show in \cite{KMOT3} that the matrix $M_0$ satisfies the
condition (1) of Definition \ref{def:H reflection}.
The matrices $M_0$ and $\mathcal H$ satisfy 
the condition (2) of Definition \ref{def:H reflection}
by the equality \eqref{eq:keep-IM1}.
\end{proof}

The following proposition will also be proved in \cite{KMOT3}.  
\begin{proposition}
\label{comoponent of van cycle}
Let $\mathbf{v}$ be a reflection vector of $M_0$ and
$$
\mathbf{v}=\sum_{J\in [p]^m}{\vcc}\quad
({\vcc} \in \mathcal V_J)
$$
be the simultaneous eigen decomposition of ${\bf v}$ with respect to
$M_1$,$\dots$,$M_m$.
Then
${\vcc}$ 
is a non-zero vector for any 
$J \in [p]^{m}$. As a consequence, $\mathcal V_J$ is a $1$-dimensional
vector space generated by ${\vcc}$.
\end{proposition}
\begin{definition}[The vcc basis]
\label{VCC basis of V}
We fix a reflection vector ${\bf v}$ of $M_0$ in $\mathcal V$.
The set $\{\vcc\}_{J\in [p]^m}$ of $\mathcal V$ is called the 
vcc (vanishing cycle component) basis
associated with ${\bf v}$.
\end{definition}
Let $\mathcal H'_j$ be the intersection matrix on the $p^{m-1}$-dimensional
vector space $\mathcal V_j'$ 
arising from the twisted homology theory 
for $\cF_{C}^{p,m-1}(a(j),B')$, and $\mathcal H_j$ be given in
\eqref{def of H_j}.
In the following proposition, we prove that
$\mathcal H'_j$ coincides 
with
$\mathcal H_j$ up to constant via the isomorphism 
$\varrho_*$ in \eqref{isomorphism with monodromy action}.
\begin{proposition}
\label{inductive intersection form}
\begin{enumerate}
\item
The matrix $M'_{0,j}$  
is an $\mathcal H'_j$-reflection, and $N_{0,j}$ is
an $\mathcal H_j$-reflection.
\item
Let ${\bf v}$ be a reflection vector of $M_0$ and set 
$$
{\bf v}_j=P_j({\bf v})\in K^{[p]^{m-1}}.
$$
Then ${\bf v}_j$ is a reflection vector
of $N_{0,j}$. As a consequence, ${\bf v}_j$ is the
reflection vector for $M'_{0,j}$.
\item
Under the isomorphism $\eqref{isomorphism with monodromy action}$,
the matrices $\mathcal H_j$ and
$\mathcal H_j'$ are
equal up to a non-zero constant multiple.

\end{enumerate}
\end{proposition}
\begin{proof}
(1) 
As for the action $M'_{0,j}$, it is a consequence of 
Lemma \ref{lem:1-eigen-space}.
Since $\mathcal H$ is invertible, so is the $\mathcal V_j$ component $\mathcal H_j$.
The action $N_{0}$ is the image of the action of the fundamental group
on $\mathcal V$, it satisfies the condition (2) of Definition 
\ref{def:H reflection}.
Therefore the restriction $N_{0,j}$ of $N_0$ to 
$\mathcal V_j$ also satisfies the same condition for the matrix
$\mathcal H_j$.

The image of the action $M_{0,j}'-I_{p^{m-1}}$ on $\mathcal V'_j$ is one dimensional
by Lemma \ref{lem:1-eigen-space}.
Therefore by Proposition \ref{comparizon of gamma factor} (2),
the dimension of the image of $N_{0,j}-I_{p^{m-1}}$ is also $1$,
and we have the condition (1) of Definition \ref{def:H reflection}.

\smallskip
\noindent
(2)
To prove that ${\bf v}_j=P_j({\bf v})$ is a reflection 
vector of the $\mathcal H_{j}$-reflection $N_{0,j}$,
it is enough to
show the inclusion
\begin{equation}
\label{fixed by N0j}
\{w\in K^{[p]^{m-1}}\,
\mid w\mathcal H_j \tr{\bf v}_j^{\vee}=0\}
\subset \{w\in K^{[p]^{m-1}}\mid {w}N_{0,j}={w}\}
\end{equation}
in $K^{[p]^{m-1}}$. 
Let ${w}\in K^{[p]^{m-1}}$ be an element of the left
hand side of \eqref{fixed by N0j}.
Using the orthogonality 
of $\mathcal V_j$ and $\mathcal V_{j'}$ for $j\neq j'$
(Proposition \ref{orthogonality and diagonal}), 
we have
\begin{equation}
\label{reflection for vj} 
E_j(w)\mathcal H \tr{\bf v}^{\vee}=
E_j(w)\mathcal H \big(\sum_{j'=1}^pE_{j'}(\tr{\bf v}_{j'}^{\vee})\big)=
w\mathcal H_j \tr{\bf v}_j^{\vee}=0.
\end{equation}
By the property of the refection vector $\bf v$ of $M_0$ 
and the equality \eqref{reflection for vj},
we have $E_j(w)(M_0-I_{p^m})=0$ and $E_j(w)M_0=E_j(w)$.
Moreover, since $E_j(w)$
is contained in the $\beta_{j,m}^{-1}$ eigenspace of $M_m$, 
we have
$$
E_j(w)M_m^i M_0 M_m^{-i}=
\beta_{j,m}^{-i}E_j(w)M_0 M_m^{-i}=
\beta_{j,m}^{-i}E_j(w)M_m^{-i}=E_j(w),
$$
and
$$
E_j(wN_{0,j})=
E_j(w)M_0( M_m M_0 M_m^{-1})\cdots ( M_m^{p-1} M_0 M_m^{1-p})
=E_j(w),
$$ 
using the equality
\eqref{eq:reduction}.
Thus the inclusion \eqref{fixed by N0j} follows.

\smallskip
\noindent
(3) 
By (2), ${\bf v}_j$ is the image of 
$N_{0,j}-I_{p^{m-1}}=M_0'-I_{p^{m-1}}$. We choose ${\bf v}_j^*$ 
so that ${\bf v}^*_j(N_{0,j}-I_{p^{m-1}})={\bf v}_j$.
Since
${u}N_{0,j}={u}M_{0,j}'$ for $u\in K^{[p]^{m-1}}$and the formula of 
an $H$-reflection, we have
$$
{u}+
\dfrac{{u}\mathcal H_j\tr {\bf v}_j^{\vee}}
{{\bf v}_j^*\mathcal H_j\tr {\bf v}_j^{\vee}}{\bf v}_j =
{u}+
\dfrac{{u}\mathcal H_j'\tr {\bf v}_j^{\vee}}
{{\bf v}_j^*\mathcal H'_j\tr {\bf v}_j^{\vee}}{\bf v}_j,
$$
and
$$
\mathcal H_j\tr {\bf v}_j^{\vee}
=\dfrac{{\bf v}^*_j\mathcal H_j\tr {\bf v}_j^{\vee}}
{{\bf v}^*_j\mathcal H'_j\tr {\bf v}_j^{\vee}}\mathcal H_j'
\tr {\bf v}_j^{\vee}.
$$ 
Since $\mathcal H_j$ and $\mathcal H_j'$ are diagonal matrices
and the $J'(j)$-th components of ${\bf v}_j$ are nonzero,
$\mathcal H_j$ is a non-zero constant 
multiple of $\mathcal H_j'$.
\end{proof}

\section{Explicit forms of circuit matrices}
\subsection{Main Theorem}
In this section, we give explicit computations of $M_0$.
We use the identification
$\mathcal V\simeq K^{p^m}$ in \eqref{hypergeom identification}. 
Let $\mathcal H$ be the intersection matrix 
defined in \eqref{intersection matrix total}. 
\begin{definition}
\label{def:h_J in cH}
For $J\in [p]^m$, we define $h_J$ by
\begin{equation}
\label{eq:Hpm}
h_J
=h_J(a,B)
=
\dfrac{
\prod_{i=1}^p (\a_i-{\textstyle{\prod\limits_{k=1}^m \b_{j_k,k}}})
}
{
\prod_{k=1}^m \big(-\b_{j_k,k}{\textstyle{\prod\limits_{1\le i\le p}^{i\ne j_k}
(\b_{i,k}-\b_{j_k,k})}}\big)
},
\end{equation}
where 
$\a_i=\ex(a_i)$, 
$\b_{i,k}=\ex(b_{i,k})$, 
and $\b_{p,k}=1$.
We define a diagonal matrix $H=\diag(h_J)_{J \in [p]^m}$.
\end{definition}
\begin{remark}
In the definition of $H$-reflection, $H$ should satisfies the 
``skew symmetric condition'' 
$H=\tr H^\vee$. There exists a constant $c$ such that
$c$ times the diagonal matrix in Definition \ref{def:h_J in cH}
satisfies this ``skew symmetric condition''.
\end{remark}

\begin{proposition}
\label{reflection diagonal matrix} 
Let ${\bf v}$ and ${\vcc}$ be the elements in Definition 
\ref{VCC basis of V}.
\begin{enumerate}
 \item 
The matrix $({\vcc}\mathcal H\tr{\vccp}^{\vee})_{J,J'}$
is a constant multiple of $H$.

\item
We have
\begin{equation}
\label{length of van cycle}
\sum_{J\in [p]^m} h_J(a,B)=(-1)^{p(m-1)}-\frac{(-1)^{m-1}\prod\limits_{i=1}^p\a_i}
{\prod\limits_{k=1}^{m}(\b_{1,k}\cdots\b_{p-1,k})}.
\end{equation}
In particular, ${\bf v}\mathcal H\tr{\bf v}^{\vee}\neq 0$
as an element of $K$ and $\lambda\neq 1$.
\item
The reflection eigenvalue $\l$ of $M_0$ for $\cF_{C}^{p,m}(a,B)$
is given by 
$$
\l=(-1)^{(p-1)(m-1)}
\prod_{k=1}^{m}(\b_{1,k}\cdots\b_{p-1,k})\Big/\Big(\prod_{i=1}^p\a_i\Big).
$$
\end{enumerate}
\end{proposition}
We will prove Proposition \ref{reflection diagonal matrix} 
in \S \ref{main theorem case m=1} and
\S \ref{proof of main theorem big m}. 
Assuming this proposition,
we state our main theorem 
 and prove it.
\begin{theorem}[Main Theorem]
\label{theorem:main theorem}
Let $\{\vcc\}$ be the vcc basis of $K^{[p]^m}$ defined in Definition
\ref{VCC basis of V}.
We define elements $v_J$ in $K$ so that
${\vcc}=v_J{\bf e}_J$ for $J\in [p]^m$,
and set $V=\diag(v_J)_J$.
Then we have the 
circuit matrices
\begin{align}
\label{M0 formula for vcc base}
&M_0
=V^{-1}\bigg(I_{p^m}-(1-\l)
\dfrac
{H}
{\sum_J h_J}
\one_{p^m}
\tr\one_{p^m}\bigg)V,
\\
\label{Mk formula for vcc base}
&M_k=\diag(\beta^{-1}_{j_k,k})_{J=(j_1,\dots, j_m)}
\quad (k=1, \dots, m),
\end{align}
where $\one_{p^m}=\tr(1,\dots, 1)\in K^{[p]^m}$.
 \end{theorem}
\begin{proof}
Change the basis $\{{\bf e}_J\}$ consisting of eigenvectors
of $M_k$ with $\{\vcc\}$. 
Then we have the equality \eqref{Mk formula for vcc base}
from Proposition \ref{orthogonality and diagonal 1}.

We show the equality \eqref{M0 formula for vcc base}.
Since ${\vcc}\mathcal H\tr{\vccp}^{\vee}=0$
for $J\neq J'$, the equalities
$$
{\bf v}\mathcal H\tr{\bf v}^{\vee}=
\sum_J{\vcc}\mathcal H\tr{\vcc}^{\vee},\quad
{\vcc}\mathcal H\tr{\bf v}^{\vee}=
{\vcc}\mathcal H\tr{\vcc}^{\vee}
$$
hold.
By the equality 
\eqref{length of van cycle},
we have $\sum_J h_J(a,B)\neq 0$, which implies
${\bf v}\mathcal H\tr{\bf v}^\vee=\sum_J{\vcc}
\mathcal H\tr{\vcc}^\vee\neq 0$ 
by Proposition \ref{reflection diagonal matrix} (1).
Thus $\lambda$ is different from $1$ by 
Proposition \ref{prop: H-reflection} (4).
By using 
Proposition \ref{reflection diagonal matrix} (1) again,
we have
\begin{align*} 
{\vcc}M_0
=&{\vcc}-(1-\l)
\dfrac
{{\vcc}\mathcal H\tr{\bf v}^{\vee}}
{{\bf v}\mathcal H\tr{\bf v}^{\vee}}
{\bf v}
\\
=&{\vcc}-(1-\l)
\dfrac
{{\vcc}\mathcal H\tr{\vcc}^{\vee}}
{\sum_J{\vcc}\mathcal H\tr{\vcc}^{\vee}}
(\sum_J{\vcc})
\\
=&{\vcc}-(1-\l)
\dfrac
{h_J}
{\sum_J h_J}
\tr\one_{p^m}V
\end{align*}
for $J \in p^{[m]}$.
The formula \eqref{M0 formula for vcc base} follows 
from the above equality and the equation ${h_J}={\bf e}_JH \one_{p^m}$. 
\end{proof}

\subsection{Case $m=1$}
\label{main theorem case m=1}
We prove Proposition \ref{reflection diagonal matrix}  
by induction on $m$.
For the case $m=1$, $\cF_C^{p,1}(a,B)$
is nothing but the generalized hypergeometric equation.
There are several approaches to study 
the monodromy representation of $\cF_C^{p,1}(a,B)$, refer to 
\cite{BH}, \cite{Ma2} and the references therein.   
Here we remark that $\pi_1(X,\dot x)$ is generated by $\rho_1$ and $\rho_0$
in case $m=1$, and that $\rho_1$ is a loop in $\C-\{0,1\}$ turning once 
positively around $x=0$, $\rho_0$ is that around $x=1$. 
We base our study on the following fact. 
\begin{fact}[{\cite[Theorem 3.7]{Ma2}}]
\label{fact:monod:m=1}
Suppose the non-integrality conditions \eqref{eq:non-integral}, 
\eqref{eq:non-integral-addition}
for $m=1$ and $\ds{\sum_{i=1}^p(a_i-b_i)\notin \Z}$. 
Here we regard $b_{1,1}$,$\dots$, $b_{p-1,1}$,$b_{p,1}$ as 
$b_1$,$\dots$,$b_{p-1},1.$ 
Then the circuit matrix $M_0$ is an $\mathcal H$-reflection
with respect to the intersection pairing.
By choosing a suitable reflection vector ${\bf v}$ of $M_0$,
we have
$E_i({\bf v}_i) \mathcal H \tr E_j({\bf v}_j)^{\vee}=0$ for $i\neq j$,
and $E_j({\bf v}_j) \mathcal H \tr E_j({\bf v}_j)^{\vee}$ is 
a constant multiple of 
\begin{align*}
h_j 
=h_j(a,b)=\frac{\prod\limits_{i=1}^p(\a_i-\b_{j})}
{-\b_{j}\prod\limits_{1\le i\le p}^{i\ne j} (\b_{i}-\b_{j})}
\end{align*}
for $1\le j\le p$, 
where ${\bf v}=\sum_{i=1}^pE_i({\bf v}_i)$ is the eigen decomposition
with respect to the action of $r(M_1)$
given in Proposition \ref{inductive intersection form} (2).
The circuit matrices of $\rho_1$ and $\rho_0$ are
$$
M_1=
\diag(\b_1^{-1},\dots,\b_{p-1}^{-1},1),
\quad 
M_0
=V^{-1}\bigg(I_p-\frac{(1-\l)H}{\sum_{i=1}^p h_i}\one_p\tr\one_p\bigg) V,
$$
where ${\vcc}=v_J{\bf e}_J\in K^{[p]}$
and
\begin{align*}
V=\diag(v_J)_J
,\quad
\l=\big(\prod\limits_{j=1}^{p-1}\b_j\big)
\big/\Big(\prod\limits_{i=1}^p \a_i\big).
\end{align*}
\end{fact}

\subsection{Computation of the intersection matrix
}
\label{proof of main theorem big m}

\begin{proof}[Proof of Proposition \ref{reflection diagonal matrix} (1)]
We use induction on $m$. 
By Proposition \ref{inductive intersection form} (3), for each fixed $j$, two vectors
$(\vccpj\mathcal H\tr \vccpj^{\vee})_{J'}$ and
$(\vccp\mathcal H'_{j}\tr \vccp^{\vee})_{J'}$ are equal
up to constant multiple. 
By induction on $m$, we assume that
$(\vccp \mathcal H'_{j}\tr \vccp^{\vee})_{J'}$ is a constant multiple
of $(h_{J'}(a(j),B'))_{J'}$.

We set $\a(j)_i=\a_i/\b_{j,m}$.
For $J'=(j_1, \dots, j_{m-1})\in [p]^{m-1}$, by the definition of $h_J$ in \eqref{eq:Hpm},
we have
\begin{align*}
&h_{J'(j)}(a,B)
=
\dfrac{
\prod_{i=1}^p 
\b_{j,m}(\dfrac{\a_i}{\b_{j,m}}-{\textstyle{\prod\limits_{k=1}^{m-1} \b_{j_k,k}}})
}
{
\prod_{k=1}^m \big(-\b_{j_k,k}{\textstyle{\prod\limits_{1\le i\le p}^{i\ne j_k}
(\b_{i,k}-\b_{j_k,k})}}\big)
}
\\
=&
\frac{-\b_{j,m}^{p-1}}
{{\textstyle{\prod\limits_{1\le i\le p}^{i\ne j}
(\b_{i,m}-\b_{j,m})}}
}
\cdot
\dfrac{
\prod_{i=1}^p 
(\a(j)_i-{\textstyle{\prod\limits_{k=1}^{m-1} \b_{j_k,k}}})
}
{
\prod_{k=1}^{m-1} \big(-\b_{j_k,k}{\textstyle{\prod\limits_{1\le i\le p}^{i\ne j_k}
(\b_{i,k}-\b_{j_k,k})}}\big)
}
\\
=&
\frac{-\b_{j,m}^{p-1}}
{{\textstyle{\prod\limits_{1\le i\le p}^{i\ne j}
(\b_{i,m}-\b_{j,m})}}
}\cdot h_{J'}(a(j),B').
\end{align*}
Since the first factor of the last term does not depend on
$J'$, 
\begin{align}
\label{up to constant I}
&\text{$(\vccpj \mathcal H \tr\vccpj^{\vee})_{J'}$
is a constant multiple of $(h_{J'}(a(j),B))_{J'}$} 
\end{align}
by the assumption of the induction.
We consider the infinitesimal inclusion defined by 
$$
(x_2, \dots, x_m)\mapsto (\e, x_2, \dots, x_m),
$$ 
and the restriction of hypergeometric series to $x_1=0$.
By the same argument considering
the action of $r(M_1)$, for $B''=({\bf b}_2, \dots, {\bf b}_m)$,
we can show that
\begin{align}
\label{up to constant II}
&\text{the matrix $(\vccpp\mathcal H \tr \vccpp^{\vee})_{J''}$ is
a constant multiple of}
\\
\nonumber
&\text{$(h_{J''(p)}(a,B''))_{J''}$, where $J''(p)=(p,j_2, \dots, j_m)$
}
\\
\nonumber
&\text{ for $J''=(j_2, \dots, j_m)$. 
} 
\end{align}
By \eqref{up to constant I}
and \eqref{up to constant II},
$(\vcc\mathcal H \vcc)_{J}$ is
a constant multiple of $(h_{J}(a,B))_{J}$.
\end{proof}
Before proving Proposition \ref{reflection diagonal matrix} (2), 
we prepare several formulas.
\begin{lemma}
\label{Lagrange formula}
For $p\ge 2$ and $-2\le r\le p$ and $\b_1,\dots, \b_p\in \C$, we have 
\begin{equation}
\label{eq:sum-formula}
\begin{array}{ll}
&\displaystyle{\sum_{k=1}^{p} \frac{\b_{k}^{r}}
{\prod\limits_{1\le i\le p}^{i\ne k} (\b_{k}-\b_{i})}}\\
=&\left\{
\begin{array}{ccl}
(-1)^{p-1}(1/\b_1+\cdots+1/\b_{p})/(\b_0\b_1\cdots\b_{p-1})
&  \textrm{if} &r=-2,\\
(-1)^{p-1}/(\b_1\b_1\cdots\b_{p}) & \textrm{if} &r=-1,\\ 
0 & \textrm{if} &0\leq r< p-1,\\ 
1 & \textrm{if} &r=p-1.\\ 
\end{array}
\right.
\end{array}
\end{equation}
\end{lemma}
\begin{proof}
A polynomial $f(x)$ of degree less than or equal to $p-1$ 
is expressed as  
\begin{equation*}
f(x)=\sum_{k=1}^p \frac{\prod\limits_{1\le i\le p}^{i\ne k} 
(x-\b_{i})}
{\prod\limits_{1\le i\le p}^{i\ne k} (\b_{k}-\b_{i})}
f(\b_{k})
\end{equation*}
by the Lagrange interpolation formula.
By putting $f(x)=x^{r}$ for $r=0, \dots, p-2, p-1$ 
and comparing the coefficients of $x^{p-1}$,
we have the statement for $r=0, \dots, p-2, p-1$.

We set 
$$
\varphi(x)=\prod_{i=1}^p(x-\beta_i)=\sum_{i=0}^pu_ix^i,
$$
and define polynomials $h_1(x), h_2(x)$ by
\begin{align*}
h_1(x)&=\dfrac{\varphi(x)-u_0}{x},\quad
h_2(x)=\dfrac{x\varphi(x)-(u_0/u_1)\varphi(x)+(u_0^2/u_1)}{x^2}.
\end{align*}
Then we have 
$$
h_1(\beta_k)=-u_0\beta_k^{-1},\quad h_2(\beta_k)=(u_0^2/u_1)\beta_k^{-2},
$$
and $\deg(h_1)=\deg(h_2)=p-1$.
We apply the Lagrange interpolation formula to $h_1(x)$ and $h_2(x)$,
then they are expressed as
\begin{align*}
h_1(x)=\sum_{k=1}^p \frac{-u_0\prod\limits_{1\le i\le p}^{i\ne k} 
(x-\b_{i})}
{\b_k\prod\limits_{1\le i\le p}^{i\ne k} (\b_{k}-\b_{i})},\quad
h_2(x)=\sum_{k=1}^p \frac{u_0^2\prod\limits_{1\le i\le p}^{i\ne k} 
(x-\b_{i})}
{u_1\b_k^2\prod\limits_{1\le i\le p}^{i\ne k} (\b_{k}-\b_{i})}.
\end{align*}
Therefore, we have
\begin{align*}
-\dfrac{1}{u_0}=\sum_{k=1}^p 
\frac{1
}
{\b_k\prod\limits_{1\le i\le p}^{i\ne k} (\b_{k}-\b_{i})},\quad
\dfrac{u_1}{u_0^2}=\sum_{k=1}^p \frac{1
}
{\b_k^2\prod\limits_{1\le i\le p}^{i\ne k} (\b_{k}-\b_{i})}
\end{align*}
by comparing the coefficients of $x^{p-1}$.
\end{proof}
\begin{proof}[Proof of 
Proposition \ref{reflection diagonal matrix}(2)]
By the assumption of the induction for $m-1$, we have
\begin{align*}
\sum_{J'}h_{J'}(a(j),B')
= 
(-1)^{p(m-2)}-\frac{(-1)^{m-2}\prod\limits_{i=1}^p\a_i}
{\b_{j,m}^p\prod\limits_{k=1}^{m-2}(\b_{1,k}\cdots\b_{p-1,k})}.
\end{align*}
It yields that 
\begin{align*}
&\sum_{J}h_{J}(a, B)
=
\sum_{j=1}^p
\frac{-\b_{j,m}^{p-1}}
{{\textstyle{\prod\limits_{1\le i\le p}^{i\ne j_k}
(\b_{i,m}-\b_{j,m})}}
}
\sum_{J'}h_{J'}(a(j),B') 
\\
=&
\sum_{j=1}^p
\frac{-(-1)^{p(m-2)}\b_{j,m}^{p-1}}
{{\textstyle{\prod\limits_{1\le i\le p}^{i\ne j_k}
(\b_{i,m}-\b_{j,m})}}
}
+\frac{(-1)^{m-2}\prod\limits_{i=1}^p\a_i}
{\prod\limits_{k=1}^{m-1}(\b_{1,k}\cdots\b_{p-1,k})}
\sum_{j=1}^p
\frac{1}
{\b_{j,m}{\textstyle{\prod\limits_{1\le i\le p}^{i\ne j_k}
(\b_{i,m}-\b_{j,m})}}
}
\\
=&
(-1)^{p(m-1)}-\frac{(-1)^{m-1}\prod\limits_{i=1}^p\a_i}
{\prod\limits_{k=1}^{m}(\b_{1,k}\cdots\b_{p-1,k})}.
\end{align*}
As for the last line of the above equalities, we use Lemma \ref{Lagrange formula}.
Therefore the formula \eqref{length of van cycle}
is valid for $m$ under the assumption of the induction.
By Proposition \ref{prop: H-reflection} (4), we have $\lambda\neq 1$.
\end{proof}

\subsection{The reflection eigenvalue of $M_0$}
\begin{definition}[Amplified vanishing cycle space ${\bf U}_p$
]
Let ${\bf v}$ be a reflection vector of $M_0$ and 
$$
{\bf v}=\sum_{j=1}^p E_j({\bf v}_j),\quad {\bf v}_j=P_j({\bf v})
$$
be the decomposition of ${\bf v}$ with respect to the 
eigenspaces $\mathcal V_j$ of $M_m$. 
We define the amplified vanishing cycle space for $M_m$ as the 
$p$-dimensional $K$-vector
subspace ${\bf U}={\bf U}_p$ spanned by $E_1({\bf v}_1),\dots,
E_p({\bf v}_p)$.
\end{definition}

\begin{proposition}
The amplified vanishing cycle space ${\bf U}$ is stable 
under the right actions of
$M_0$ and $M_m$.
\end{proposition}
\begin{proof}
For $j=1,\dots, p$, we have $E_j({\bf v}_j)M_m=\beta_{j,m}^{-1}
E_j({\bf v}_j)\in {\bf U}$ and 
$$
E_j({\bf v}_j)M_0=E_j({\bf v}_j)+\gamma {\bf v}
=E_j({\bf v}_j)+\gamma(E_1({\bf v}_1)+\cdots +E_p({\bf v}_p))
$$ 
for some $\gamma\in K$. Note that $E_j({\bf v}_j)M_0\in {\bf U}$.
\end{proof}
We set $W=M_0M_m$. 
We choose a basis $\{u_i\}_{i\in [p]}$ of ${\bf U}$.
Since ${\bf U}$ is stable under the right
action of $W$, 
there exists a $p\times p$ matrix $(a_{ij})$ such that 
$$
u_iW=\sum_{j\in [p]}a_{ij}u_j.
$$ 
We define the trace $\mathrm{tr}(W|_{\bf U})$
of the right action of matrix $W$ restricted to 
${\bf U}$ by 
$\sum_{i\in[p]}a_{ii}$.
It is easy to see that $\mathrm{tr}(W|_{\bf U})$ is
independent of the choice of the basis $\{u_i\}_{i\in [p]}$ of ${\bf U}$.
\begin{proposition}
\label{trace free for W}
\begin{enumerate}
\item
The right action of $W^p$ 
commutes with those of $M_0$ and $M_m$ on ${\bf U}$. 
Moreover, $W^p$ is a scalar action on ${\bf U}$. 
\item 
The trace $\mathrm{tr}(W|_{\bf U})$ 
is $0$.
\end{enumerate}
\end{proposition}
\begin{proof}
(1) 
Since $\bold v$ is a reflection vector of $M_0$, we have
$$
{\bf v}W^pM_0 = {\bf v}W^p+\gamma {\bf v},\quad
{\bf v}M_0W^p=\lambda {\bf v}W^p
$$
for some $\gamma \in K$. 
By the relation \eqref{eq:pi1-relations}
in Corollary \ref{circular relation}, 
the matrices $M_0$, $M_m$ and $W$ satisfy 
$$
W^pM_0=(M_0M_m)^pM_0=M_0(M_mM_0)^p= M_0(M_0M_m)^p=M_0W^p,
$$
which means 
that $W^p$ and $M_0$ commute. Therefore, we have
$$
{\bf v}W^p+\gamma {\bf v}={\bf v}W^pM_0={\bf v}M_0W^p=\lambda {\bf v}W^p.
$$
By Proposition \ref{reflection diagonal matrix} (2), 
$\lambda$ is different from $1$, and
\begin{equation}
\label{p-th power is constant}
{\bf v}W^p=\lambda' {\bf v},\quad \lambda'=\dfrac{\gamma}{\lambda-1}
\neq 0. 
\end{equation}

Since $N_0$ commutes with $M_m$ by Theorem \ref{th:reduction} (2), 
$W^p$ and $M_m$ commute by restricting the following equality to 
$\mathbf{U}$:
\begin{align*}
N_0M_m^{p}
=&M_0(M_mM_0M_m^{-1})\cdots (M_m^{p-1}M_0M_m^{1-p})M_m^{p}
=(M_0M_m)^p.
\end{align*}
Therefore the action of
$W^p$ preserves the eigen decomposition of $\mathbf{U}_p$.
Thus, for $j\in [p]$, 
there exists $c_j\in K$ such that 
\begin{equation}
\label{eq:eigenvalue-cj} 
E_j({\bf v}_j)W^p=c_jE_j({\bf v}_j).
\end{equation} 
By comparing \eqref{p-th power is constant} with \eqref{eq:eigenvalue-cj}, 
we have  
$c_j=\lambda'$,
and the right action of $W^p$ restricted to ${\bf U}$
is a scalar multiple of $\lambda'$.

\smallskip
\noindent
(2) We claim that the trace $\mathrm{tr}(W\mid_{{\bf U}})$ of $W$ is $0$. 
To show this claim, we see that the $p$ vectors 
${\bf v}$, ${\bf v}(M_0M_m)$, $\dots$, ${\bf v}(M_0M_m)^{p-1}$
span the space $\mathbf{U}$. 
By the expression of $M_0$ in \eqref{H reflection}, there 
exist $c_1,c_2,\dots,c_{p-2}\in \C$ such that 
\begin{align*}
{\bf v}(M_0M_m)&=\l({\bf v}M_m),\\
{\bf v}(M_0M_m)^2&
=(\l{\bf v}M_m)(M_0M_m)=(\l{\bf v}M_m+c_1{\bf v})M_m\\
&=\l({\bf v}M_m^2)+c_1({\bf v}M_m),\\
{\bf v}(M_0M_m)^3&=(\l({\bf v}M_m^2)+c_1({\bf v}M_m))(M_0M_m)\\
&=(\l({\bf v}M_m^2)+c_1({\bf v}M_m)+c_2{\bf v})M_m\\
&=\l({\bf v}M_m^3)+c_1({\bf v}M_m^2)+c_2({\bf v}M_m),\\
&\hspace{2mm} \vdots \\
{\bf v}(M_0M_m)^{p-1}
&=\l(\!{\bf v}M_m^{p-1})\!+\!c_1(\!{\bf v}M_m^{p-2})\!+\!\cdots
\!+\!c_{p-2}(\!{\bf v}M_m).
\end{align*}
Since ${\bf v}$, 
${\bf v}M_m$ $\dots$, ${\bf v}M_m^{p-1}$ are linearly independent
under the condition \eqref{eq:non-integral-addition}, 
${\bf v}$, 
${\bf v}(M_0M_m)$, $\dots$, ${\bf v}(M_0M_m)^{p-1}$ span the 
space $\mathbf{U}$. We compute $\mathrm{tr}(W\mid_{{\bf U}})$
using the basis $\{{\bf v}(M_0M_m)^{i}\}_{i=0,\dots,p-1}$
and \eqref{p-th power is constant}.
We have
$$
{\bf v}(M_0M_m)^{i}(M_0M_m)
=\begin{cases}
{\bf v}(M_0M_m)^{i+1} &\text{ for }0\leq i\leq p-2,
\\
\l'{\bf v} &\text{ for }i = p-1,  
 \end{cases}
$$
where $\lambda'$ is given in the proof of (1).
The diagonal entries with respect to the above basis
are zero, we have $\Tr(W|_{{\bf U}})=0$.
\end{proof}
\begin{proof}[The proof of Proposition \ref{reflection diagonal matrix} (3)]
 We compute the trace of $W$ on ${\bf U}$.
By Proposition \ref
{prop: H-reflection} (3) and 
Proposition \ref{reflection diagonal matrix} (2),
$\lambda$ is different from $1$. 
We set 
$S_i=\sum_{J'\in[p]^{m-1}}h_{J'(i)}$ and
$S=\sum_{i=1}^p S_i$. 
Then we have
$$
S=(-1)^{p(m-1)}-\frac{(-1)^{m-1}\prod\limits_{i=1}^p \a_i}
{\prod\limits_{k=1}^m(\b_{1,k}\cdots \b_{p,k})}
$$
by Proposition \ref{reflection diagonal matrix} (2),
and
$$
\dfrac{S_i}{S}=
\dfrac
{E_i({\bf v}_i)\mathcal H\tr{\bf v}_i}
{{\bf v}\mathcal H\tr{\bf v}}=
\dfrac
{E_i({\bf v}_i)\mathcal H\tr{\bf v}}
{{\bf v}\mathcal H\tr{\bf v}}
$$
by Proposition \ref{reflection diagonal matrix} (1).
Since
\begin{align*}
&E_i({\bf v}_i) M_0M_m
=E_i({\bf v}_i)M_m-(1-\lambda)
\dfrac
{E_i({\bf v}_i)\mathcal H\tr{\bf v}}
{{\bf v}\mathcal H\tr{\bf v}}
(\sum_{j=1}^pE_j({\bf v}_j)M_m)
\\
=&E_i({\bf v}_i)\b_{i,m}^{-1}-(1-\lambda)
\dfrac{S_i}{S}
(\sum_{j=1}^pE_j({\bf v}_j)\b_{j,m}^{-1})
\\
=&E_i({\bf v}_i)
\Big(1
+(\lambda-1)
\dfrac{S_i}{S}\Big)\b_{i,m}^{-1}
+(\lambda-1)\sum_{1\leq j\leq p, j\neq i}
\dfrac{S_i}{S}(E_j({\bf v}_j)\b_{j,m}^{-1}),
\end{align*}
the trace $\Tr(W|_{\bf U})$ is equal to 
\begin{align*}
\Tr(W|_{\bf U})
=\sum_{i=1}^p
\Big(1
+(\lambda-1)
\dfrac{S_i}{S}\Big)\b_{i,m}^{-1}.
\end{align*}
Since $\Tr(W|_{\bf U})=0$ by Proposition \ref{trace free for W} (2),
the above equality implies 
\begin{align*}
\l
=&1-
\big(\frac{1}{\b_{1,m}}+\cdots +\frac{1}{\b_{p,m}}\big)
\cdot 
S\Big/
{\big(\sum\limits_{i=1}^p (S_i/\b_{i,m})\big)}.
\end{align*}
Using Lemma \ref{Lagrange formula}, we have the equality 
\begin{align*}
&\sum\limits_{i=1}^p (S_i/\b_{i,m})=
\frac{(-1)^{m-1}\prod\limits_{i=1}^p\a_i}
{\prod\limits_{k=1}^{m-1}(\b_{1,k}\cdots\b_{p,k})}\cdot \frac{-1}{
{\b_{1,m}\cdots\b_{p-1,m}}}
(\frac{1}{\b_{1,m}}+\cdots+\frac{1}{\b_{p,m}}),
\end{align*}
and as a consequence, 
$\l$ is expressed by 
\begin{align*}
\l
=&
(-1)^{(p-1)(m-1)}\frac{\prod\limits_{k=1}^m(\b_{1,k}\cdots\b_{p-1,k})}
{\prod\limits_{i=1}^p{\a_i}}
\end{align*}
as an element of $K$.  
\end{proof}

\begin{remark}
We have 
\begin{equation}
\label{eq:(1-L)/trH}
\frac{1-\l}{\sum_J h_J}
=-(-1)^{p(m-1)}\lambda
=(-1)^{m}
\big(\prod_{k=1}^{m}(\b_{1,k}\cdots\b_{p-1,k})\big)\big/\big(\prod_{i=1}^p\a_i\big).
\end{equation}
Though we have assumed
that
the eigenvalue $\l$ of $M_0$ is different from $1$ 
in Proposition \ref{reflection diagonal matrix} (2), 
the expression of $M_0$ in this proposition 
is valid even in the case $\l=1$ 
by the continuity of the monodromy representation 
of $\cF_C^{p,m}(a,B)$ with respect to the parameters $a,B$ 
together with \eqref{eq:(1-L)/trH}. 
\end{remark}

 \ifdefined\ifmaster
 \else
\end{document}
\fi

\end{document}